\documentclass[11pt,reqno]{amsart}
\usepackage{geometry}
\usepackage{xcolor}
\usepackage{amsthm}
\usepackage{amssymb}
\usepackage{hyperref}
\usepackage{arydshln}
\usepackage{bbm}
\usepackage{mathrsfs}
\usepackage{graphicx}

\usepackage[T1]{fontenc} 
\usepackage[utf8]{inputenc}
\usepackage[english]{babel}

\usepackage{float}
\usepackage{multirow}
\usepackage{caption}

\makeatletter
\numberwithin{equation}{section}
\numberwithin{figure}{section}

\newcommand{\R}{\mathbb{R}}
\newcommand{\N}{\mathbb{N}}

\newcommand{\F}{\mathcal{F}}

\renewcommand{\P}{\mathbb{P}}
\newcommand{\E}{\mathbb{E}}
\newcommand{\e}{\varepsilon}

\newcommand{\1}{\mathbbm{1}}

\newcommand{\bfrac}[2]{\genfrac{}{}{0pt}{}{#1}{#2}}
\newcommand{\Lu}{ \overset{\Theta}{\Longrightarrow}}

\newtheorem{Theorem}{Theorem}[section]
\newtheorem{Proposition}[Theorem]{Proposition}\newtheorem{Corollary}[Theorem]{Corollary}\newtheorem{Lemma}[Theorem]{Lemma}\newtheorem{Remark}[Theorem]{Remark}\newtheorem{Definition}[Theorem]{Definition}\newtheorem{Example}[Theorem]{Example}

\numberwithin{equation}{section}

\makeatother

\begin{document}
\title[Ergodicity and LLN for the VCIR process]{Ergodicity and Law-of-large numbers for the Volterra Cox-Ingersoll-Ross process}
\author{Mohamed Ben Alaya}
\author{Martin Friesen}
\author{Jonas Kremer$^1$}
\thanks{$^1$The views, opinions, positions or strategies expressed in this article are those of the authors and do not necessarily
represent the views, opinions, positions or strategies of, and should not be attributed to E.ON Energy Markets.}

\address[Mohamed Ben Alaya]{Laboratoire de Mathématiques Raphaël Salem (LMRS)
\\ Université de Rouen Normandie}
\email{mohamed.ben-alaya@univ-rouen.fr}

\address[Martin Friesen]{School of Mathematical Sciences\\
Dublin City University\\ Glasnevin, Dublin 9, Ireland}
\email{martin.friesen@dcu.ie}

\address[Jonas Kremer]{Modelling \& Quant. Analytics \\ E.ON Energy Markets \\ 45131 Essen, Germany}
\email{jonas.kremer@eon.com}

\date{\today}



\begin{abstract}
 We study the ergodic properties for the Volterra Cox-Ingersoll-Ross process on $\R_+$ and its stationary version. Based on a fine asymptotic analysis of the corresponding Volterra Riccati equation combined with the affine transformation formula, we first show that the finite-dimensional distributions are asymptotically independent. Afterwards, we prove a Law of Large Numbers in $L^p(\Omega)$ with $p \geq 2$ and show that the stationary process is ergodic. As an application, we prove the consistency of the method of moments and study the maximum-likelihood estimation for continuous and discrete high-frequency observations. 
\end{abstract}

\maketitle

\allowdisplaybreaks

\section{Introduction}

\subsection{Volterra Cox-Ingersoll-Ross process}

Stochastic Volterra processes allow for the modelling of roughness in sample paths (for an overview of the literature see e.g. \cite{math11194201}), and a flexible description of short and long-range dependencies, see \cite{MR3561100}. Both properties are captured via a Volterra kernel $K \in L_{loc}^2(\R_+)$. The fractional Riemann-Liouville kernel 
\begin{align}\label{eq: fractional kernel}
 K_{\alpha}(t) = \frac{t^{\alpha - 1}}{\Gamma(\alpha)}, \qquad \alpha \in (1/2, 1]
\end{align}
constitutes the most prominent example of Volterra kernels that allow for flexible incorporation of rough sample path behaviour. Other examples of Volterra kernels covered by this work are, e.g., the exponentially damped fractional kernel $K(t) = \frac{t^{\alpha-1}}{\Gamma(\alpha)}e^{-\lambda t}$, $K(t) = \sum_{j=1}^N c_j e^{-\lambda_j t}$ used for the Markovian approximation of solutions of \eqref{eq: general equation} \cite{MR3934104, MR4521278}, and the fractional $\log$-kernel $K(t) = \log(1+t^{-\alpha})$ with $\alpha \in (0,1]$. In the absence of jumps and dimension one, a stochastic Volterra process of convolution type is given by  
\begin{align}\label{eq: general equation}
 X_t = x_0 + \int_0^t K(t-s)b(X_s)ds + \int_0^t K(t-s)\sigma(X_s)dB_s,
\end{align}
where $(B_t)_{t \geq 0}$ denotes a one-dimensional standard Brownian motion, $K \in L_{loc}^2(\R_+)$ the Volterra kernel, and $b, \sigma: \R \longrightarrow \R$ are the drift and diffusion coefficients. 

Processes of the form \eqref{eq: general equation} provide a feasible framework for the modelling of rough volatility in Mathematical Finance, see e.g. \cite{MR3494612, MR3778355, MR4188876, MR3805308}. If the coefficients $b$ and $\sigma^2$ are affine linear in the state variable, then \eqref{eq: general equation} is called \textit{affine Volterra process}. In contrast to general stochastic Volterra processes, an important advantage of affine Volterra processes lies in their analytical traceability, see \cite{MR1994043, MR3905737}. For such processes, many computations can be carried out semi-explicit (e.g. via the affine transformation formula), making affine Volterra processes feasible for particular applications such as option pricing, see e.g. \cite{MR4019885, MR4412586, MR3905737, MR4091168}. We focus on Volterra kernels that satisfy the following set of regularity assumptions uniformly in $0 \leq t-s \leq 1$, which provide the existence and uniqueness in law for the process and its stationary version.
\begin{enumerate}
    \item[(K1)] The Volterra kernel $K \in L_{loc}^2(\R_+)$ is completely monotone, not identically zero, and there exists $\gamma \in (0,1]$ such that 
    \begin{align}\label{eq: K global regularity}
    \sup_{0 \leq t-s \leq 1}(t-s)^{-2\gamma}\left(\int_s^t K(r)^2 dr + \int_0^{\infty}(K((t-s) + r) - K(r))^2 dr\right) < \infty.
    \end{align}    
    \item[(K2)] There exists $\eta \in (0,1)$ such that 
    \begin{align*}
    \int_0^T t^{-2\eta}K(t)^2 dt + \int_0^T \int_0^T \frac{|K(t)-K(s)|^2}{|t-s|^{1+2\eta}}dsdt < \infty, \qquad T > 0.
    \end{align*}
\end{enumerate}

Suppose that condition (K1) is satisfied and $x_0, b, \sigma \geq 0$, $\beta \in \R$. The \textit{Volterra Cox-Ingersoll-Ross process} (VCIR process) is the affine Volterra process on state-space $\R_+$ obtained as the unique weak solution $X$ on $\R_+$ of the stochastic Volterra equation
\begin{align}\label{eq: VCIR}
    X_t = x_0 + \int_0^t K(t-s)\left( b + \beta X_s \right)ds + \sigma \int_0^t K(t-s)\sqrt{X_s}dB_s
\end{align}
defined on some stochastic basis $(\Omega, \F, (\F_t)_{t\geq 0}, \mathbb{P}_{b,\beta})$ with the usual conditions supporting a standard Brownian motion $(B_t)_{t \geq 0}$, see \cite[Theorem 6.1]{MR4019885}. For each $\eta \in (0,\gamma)$, $X$ admits a modification with $\eta$-H\"older continuous sample paths and satisfies $\sup_{t \in [0, T]}\E_{b,\beta}[|X_t|^p] < \infty$ for each $p \geq 2$ and $T >0$. Similarly to the classical CIR process (that is $K(t) = 1$), also the VCIR process satisfies an affine transformation formula, see \cite[Theorem 6.1]{MR4019885}. Under the additional condition (K2), a more general version for the Laplace transform of $\int_{[0,t]}X_{t-s}\mu(ds)$ with $\mu$ a locally finite measure on $\R_+$ was given in \cite[Corollary 3.11]{FJ22} which reads as 
\begin{align}\label{eq: VCIR Laplace}
    & \E_{b,\beta}\left[ e^{- \int_{[0,t]}X_{t-s}\mu(ds)}\right]
    \\ &\qquad \qquad = \exp\left( - x_0 \mu([0,t]) - x_0\int_0^t R(V(s;\mu))ds - b \int_0^t V(s;\mu)ds\right) \nonumber
\end{align}
where $R(x) = \beta x - \frac{\sigma^2}{2}x^2$, and $V(\cdot;\mu) \in L_{loc}^2(\R_+)$ is the unique nonnegative solution of the Riccati Volterra equation
\begin{align}\label{eq: V mu}
    V(t;\mu) = \int_{[0,t]}K(t-s)\mu(ds) + \int_0^t K(t-s)R(V(s;\mu))ds.
\end{align}
This formula allows us to express the finite-dimensional distributions of $X$ via $V(\cdot;\mu)$ for appropriate choices of $\mu$. For a detailed analysis of such Riccati-type equations, we refer to \cite[Section 6]{MR4019885}, \cite[Section 3]{FJ22}, and \cite{MR4031338}.

\subsection{Law-of-Large numbers and ergodicity}

In mathematical finance, the Volterra Cox-Ingersoll-Ross process is used for the modelling of rough volatility. Since the latter are mean-reverting, we suppose the linear drift is negative, i.e. $\beta < 0$. Consequently, following \cite[Section 5]{FJ22}, $X_t$ converges towards a limit distribution $\pi_{b,\beta}$ as $t \to \infty$. Its Laplace transform is given by
\begin{align}\label{eq: pi affine formula}
 \int_{\R_+} e^{-u x}\pi_{b,\beta}(dx) = \exp\left( - x_0u - x_0 \int_0^{\infty} R(\overline{V}(s; u))ds - b \int_0^{\infty} \overline{V}(s; u)ds\right),
\end{align}
where $\overline{V}(\cdot;u) \in L^1(\R_+) \cap L^2(\R_+)$ with $u \geq 0$ is the unique solution of \eqref{eq: V mu} with $\mu(ds) = u \delta_0(ds)$. Due to the additional memory in stochastic Volterra processes, the limiting distribution generally depends on the initial state $x_0$. It was shown that $\pi_{b,\beta}$ is independent of $x_0$ iff $K \not \in L^1(\R_+)$.

For each such limit distribution $\pi_{b,\beta}$, there also exists an underlying stationary process with time-marginals given by the limit distribution constructed via the extended affine transformation \eqref{eq: VCIR Laplace}. Namely, under conditions (K1), (K2) and $\beta < 0$, there exists a stationary probability law $\P^{\mathrm{stat}}_{b,\beta}$ on $C(\R_+)$ such that the shifted process satisfies $(X_{t+h})_{t \geq 0} \Longrightarrow \P^{\mathrm{stat}}_{b,\beta}$ weakly on $C(\R_+)$ as $h \to \infty$. The finite-dimensional distributions of $\P^{\mathrm{stat}}_{b,\beta}$ are according to \cite[Theorem 5.4]{FJ22} characterized via the affine transformation formula
\begin{align*}
 &\ \E_{b,\beta}^{\mathrm{stat}}\left[ e^{-\sum_{k=1}^n u_k X_{t_k}}\right]
 \\ &\qquad \qquad = \exp\left( - x_0\left(\sum_{k=1}^n u_k\right) - x_0\int_0^{\infty}R(V(s; \mu_{t_1,\dots, t_n})) ds - b \int_0^{\infty}V(s; \mu_{t_1,\dots, t_n})ds \right)
\end{align*}
where $n \in \N$, $u_1,\dots, u_n \geq 0$, and $V(\cdot; \mu_{t_1,\dots, t_n})$ denotes the unique solution of \eqref{eq: V mu} with the particular choice  
\begin{align}\label{eq: mu VCIR finite dimensional}
 \mu_{t_1,\dots, t_n} = \sum_{k=1}^n u_k \delta_{t_n - t_k}. 
\end{align}
Note that since $\mu_{t_1,\dots, t_n}(\R_+) < \infty$, it follows that $V(\cdot; \mu_{t_1,\dots, t_n}) \in L^1(\R_+) \cap L^2(\R_+)$ and hence all integrals are well-defined. In particular, for $n = 1$ we recover $X_t \Longrightarrow \pi_{b,\beta}$ as $t \to \infty$ together with \eqref{eq: pi affine formula}.

In this work, we study ergodic properties of the VCIR process and its stationary law $\P_{b,\beta}^{\mathrm{stat}}$. Our main result, Theorem~\ref{thm: asymptotic independence}, establishes that the VCIR process exhibits \textit{asymptotic independence} in the sense that the joint law satisfies
\begin{align}\label{eq: asymptotic independence}
 \begin{cases}(X_{t_1}, \dots, X_{t_n}) \Longrightarrow \pi_{b,\beta}^{\otimes n} \ \text{ as } \ t_1,\dots, t_n \to \infty
 \\ \text{ such that } \min_{k}|t_{k+1} - t_k| \to \infty \text{ with } 0 < t_1 < \dots < t_n.
 \end{cases}
\end{align}
The same assertion also holds for the stationary process with law $\P_{b,\beta}^{\mathrm{stat}}$. Let us remark that this convergence holds locally uniformly with respect to the drift parameters $(b,\beta)$. The definition and additional properties of uniform weak convergence are given in the appendix (see Section \ref{sec:uniform weak convergence}). Our proof of the asymptotic independence property is based on the generalised affine transformation formula \eqref{eq: VCIR Laplace} combined with a fine asymptotic analysis of the corresponding nonlinear Volterra Riccatti equation \eqref{eq: V mu}. It is important to note that, due to the lack of the Markov property on $\R_+$, this result cannot be deduced from the convergence of the one-dimensional distributions.

In our second main result, Theorem \ref{thm: law of large numbers VCIR}, we study the Law of Large Numbers for the VCIR process, and its stationary version locally uniformly in the drift parameters $(b,\beta)$. While for Markov processes the Law of Large Numbers is well-understood under weak ergodic rates (see \cite{MR3684455, Kulik}), these results are not applicable since the VCIR process is not a Markov process on $\R_+$. In particular, convergence of one-dimensional distributions alone does not contain enough information to conclude the process's Law of Large Numbers or further ergodic properties. To tackle this problem, let us first note that \eqref{eq: asymptotic independence} gives $\mathrm{cov}(X_t,X_s) \longrightarrow 0$ as $|t-s| \to \infty$ with $t,s \to \infty$. Hence by adjusting the proofs from \cite{MR3684455} to general non-Markovian settings as done in \cite[Lemma 2.1]{BFK24}, we deduce in Theorem \ref{thm: law of large numbers VCIR} the Law of Large Numbers in $L^p(\Omega, \P_{b,\beta})$, i.e., 
\begin{align}\label{eq: Birkhoff VOU}
 \frac{1}{T}\int_0^T f(X_s)ds \longrightarrow \int_{\R_+}f(x)\pi_{b,\beta}(dx)
\end{align}
locally uniformly in the parameters $(b,\beta)$ for a large class of functions $f$. To treat also the case of functions that are only a.e. continuous, we combine the uniform continuous mapping theorem (see Theorem \ref{thm: uniform CMT}) with a refined version of the Besov regularity of the limit distribution established in \cite{FJ22}. Note that the stationary process $X^{\mathrm{stat}}$ also satisfies \eqref{eq: Birkhoff VOU}. Finally, since the right-hand side of \eqref{eq: Birkhoff VOU} is deterministic, in Corollary \ref{cor: VCIR mixing} we show that $X^{\mathrm{stat}}$ is ergodic in the sense that its shift-invariant $\sigma$-algebra is trivial. To our knowledge, the latter provides the first ergodicity result for non-Gaussian Volterra processes.

\subsection{Application to statistical inference}

As an application of the established Law of Large Numbers, we study the estimation of drift parameters $(b,\beta)$ based on continuous, discrete high-frequency, and discrete low-frequency observations. The probably simplest method to estimate the drift parameters $(b,\beta)$ is based on the method of moments. Since here explicit computations are crucial to obtain \textit{explicit estimators}, in Section 2.4 we focus on the particular case of fractional kernels \eqref{eq: fractional kernel}. Applying the Law of Large Numbers \eqref{eq: Birkhoff VOU}, we derive in Corollary \ref{cor: VCIR method of moments} the consistency for the method of moments.

A more delicate estimation method that allows for more efficient estimators is provided by the maximum-likelihood method. In the absence of memory, results applicable to classical affine processes have been obtained in \cite{MR3035268, MR3216637, BK13}, while we refer to \cite{MR2144185} for the general theory on statistical inference for Markov diffusion processes. Following the approach described in \cite{BFK24} for general stochastic Volterra equations, in Section 3 we use the Law of Large Numbers to establish convergence properties for the MLE of $(b,\beta)$ first for continuous, and then for discrete high-frequency observations. More precisely, based on continuous observations we first derive an explicit formula for the MLE $(\widehat{b}_T, \widehat{\beta}_T)$ and then show that it is strongly consistent, and asymptotically normal locally uniformly in the drift parameters $(b,\beta)$. Subsequently, we treat the case of discrete high-frequency observations by numerical discretisation of the continuous MLE and, in particular, derive bounds on the discretisation error.

Finally, we conclude with additional comments on the estimation of the other model parameters $\sigma, K$. Similarly to classical diffusion processes, we can estimate $\sigma$ via a quadratic variation provided that $X$ is continuously observed on a \textit{fixed} interval $[0,T]$. It follows from representation \eqref{eq: Z process} that the process given by $Z_t(X) = (b + \beta X_t)dt + \sigma \sqrt{X_t}dB_t$ is measurable with respect to $X$. Using the quadratic variation $\langle Z(X) \rangle_T = \sigma^2 \int_0^T X_s ds$, we find the relation
\begin{align}\label{eq: sigma2 estimator}
 \sigma^2 = \frac{\langle Z(X) \rangle_T}{\int_0^T X_t dt}
\end{align}
which allows from the estimation of $\sigma$, provided that $X$ (and hence $Z(X)$ via \eqref{eq: Z process}) is continuously observed. For discrete high-frequency observations, we may consider a discretised version of \eqref{eq: Z process} and hence \eqref{eq: sigma2 estimator}. Convergence rates for such discretisations can be obtained via Lemma \ref{lemma: discretisation VOU} and \eqref{eq: Z discretisation 1}.

On the other side, the Volterra kernel $K$ can be approximated via a discretisation of its Bernstein measure \cite{MR3934104, MR4521278} which, in particular, should allow for the estimation of these parameters from given observations. For the case of rough kernels \eqref{eq: fractional kernel}, further estimation techniques have been analysed in \cite{MR4804179}. Non-parametric estimation techniques for the Volterra kernel have not yet been studied in the literature, and are left for future research.

\subsection{Structure of the work}

In section 2, we first briefly collect some auxiliary moment estimates and then prove the asymptotic independence \eqref{eq: asymptotic independence} as our main result. Afterwards, we prove the Law of Large Numbers, study the ergodicity for the stationary processes, and apply our results to the method of moments. An application of our results to the maximum-likelihood estimation is considered in Section 3 first for continuous observations and then for discrete high-frequency observations. Numerical illustrations are given in Section 4, while further technical results on the weak convergence uniformly in the parameter space are discussed in the appendix.

\section{Ergodicity for the Volterra Cox-Ingersoll-Ross process}

\subsection{Uniform bounds}

Let $K \in L_{loc}^2(\R_+)$ and $\beta \in \R$. According to \cite[Theorem 3.5]{MR1050319}, there exists a unique solution $E_{\beta} \in L_{loc}^2(\R_+)$ of the linear Volterra equation
\begin{align}\label{eq: Ebeta definition}
 E_{\beta}(t) = K(t) + \beta \int_0^t K(t-s)E_{\beta}(s)ds.
\end{align}
Moreover, since $K \in L_{loc}^2(\R_+) \cap C((0,\infty))$, one can show that $E_{\beta} \in L_{loc}^2(\R_+) \cap C((0,\infty))$. Finally, the function $R_{\beta}(t) = (-\beta)E_{\beta}(t)$ is the resolvent of the second kind as defined in \cite{MR1050319}. 

Now suppose that $\beta < 0$ and $K$ is completely monotone. It follows from \cite[Chapter 5, Theorem 3.1]{MR1050319} that $R_{\beta} = |\beta|E_{\beta} \in L^1(\R_+)$ is completely monotone. Hence $E_{\beta}$ is also completely monotone such that $E_{\beta} \in L^1(\R_+) \cap L^2(\R_+)$. Taking Laplace transforms in the definition of $E_{\beta} = K + \beta K \ast E_{\beta}$, we find 
\begin{align}\label{eq: Ebeta integral}
         \int_0^{\infty}E_{\beta}(t)dt = \frac{1}{\|K\|_{L^1(\R_+)}^{-1} + |\beta|}
        \end{align}
with the convention $1/\infty = 0$ when $K \not \in L^1(\R_+)$. Such a function satisfies the monotonicity property 
\begin{align}\label{eq: Ebeta monotonicity}
 E_{\beta} \leq E_{\widetilde{\beta}} \ \text{ for } \ \widetilde{\beta} < \beta < 0
\end{align}
as shown in \cite[Lemma B.3]{BBF23}. Let $x_0, b, \sigma \geq 0$, $\beta \in \R$ and $K$ satisfy condition (K1). Let $X$ be the VCIR process defined on some filtered probability space $(\Omega, \F, (\F_t)_{t \geq 0}, \P_{b,\beta})$ with the usual conditions. The next proposition provides uniform moment bounds with respect to time and the parameters $(b,\beta)$. 

\begin{Proposition}\label{prop: VCIR technical stuff}
    Suppose that condition (K1) holds and that $\beta < 0$. Then for each $p \geq 2$ and each compact $\Theta \subset \R_+ \times (-\infty, 0)$ there exists a constant $C_p(\Theta)$ such that
    \[
     \sup_{t \geq 0}\sup_{(b,\beta) \in \Theta}\E_{b,\beta}[ X_t^p ] \leq C_p(\Theta) < \infty
    \]
    and for all $0 \leq t-s \leq 1$ it holds that
    \[
     \sup_{(b,\beta) \in \Theta}\E_{b,\beta}[|X_t-X_s|^p] \leq C_p(\Theta)|t-s|^{p\gamma}.
    \]
    Finally, it holds that
    \begin{align*}
     \lim_{t \to \infty}&\sup_{(b,\beta) \in \Theta} \left|\E_{b,\beta}[X_t^j] - m_j(b,\beta)\right| = 0, \qquad j = 1,2
    \end{align*}
    where $m_1(b,\beta), m_2(b,\beta)$ are given by
    \begin{align*}
        m_1(b,\beta) &= \frac{x_0 + \|K\|_{L^1(\R_+)}b}{1 + \|K\|_{L^1(\R_+)}|\beta|}
        \\ m_2(b,\beta) &= m_1(b,\beta)^2 + m_1(b,\beta) \sigma^2 \int_0^{\infty} E_{\beta}(t)^2 dt.
    \end{align*}
\end{Proposition}
\begin{proof}
    Using the variation of constants formula (see e.g. \cite[Lemma 2.5]{MR4019885} combined with $R_{\beta} = |\beta|E_{\beta}$), we find 
    \begin{align}\label{eq: VCIR expectation form}
        X_t = \E_{b,\beta}[X_t] + \sigma \int_0^t E_{\beta}(t-s)\sqrt{X_s}dB_s
    \end{align}
    where the expectation is given by 
    \[
        \E_{b,\beta}[X_t] = \left( 1 + \beta \int_0^t E_{\beta}(s)ds\right)x_0 + b \int_0^t E_{\beta}(s)ds.
    \]
    Using \eqref{eq: Ebeta monotonicity}, it is easy to see that 
    \begin{align}\label{eq: Ebeta tail bound}
        \lim_{t \to \infty}\sup_{\beta \leq -1/N}\left( \int_t^{\infty}E_{\beta}(r)dr + \int_t^{\infty}E_{\beta}(r)^2dr\right) = 0.
    \end{align}
    Hence, using \eqref{eq: Ebeta integral}, it is clear that $\sup_{t \geq 0} \sup_{(b,\beta) \in \Theta}\E_{b,\beta}[X_t] < \infty$ and $\E_{b,\beta}[X_t] \longrightarrow m_1(b,\beta)$ converges uniformly on $\Theta$ as $t \to \infty$. Likewise, its variance is given by $\mathrm{var}_{b,\beta}[X_t] = \sigma^2 \int_0^t E_{\beta}(t-s)^2\E_{b,\beta}[X_s]ds$ and hence
    \[
    \E_{b,\beta}[X_t^2] = \left( \E_{b,\beta}[X_t]\right)^2 + \sigma^2 \int_0^t E_{\beta}(t-s)^2\E_{b,\beta}[X_s]ds.
    \]
    Again using \eqref{eq: Ebeta tail bound}, we conclude that the second moment is uniformly bounded for $t \geq 0$ and $(b,\beta) \in \Theta$, and converges uniformly on $\Theta$ as $t \to \infty$.    
    
    It remains to prove the uniform bounds on the $p$-moments and H\"older increments in $L^p(\Omega, \P_{b,\beta})$. Firstly, using \cite[Lemma A.3]{FJ22} and then \eqref{eq: K global regularity} from condition (K1), we find a constant $C > 0$ (independent of $\beta$) such that for all $0 \leq t-s \leq 1$ and any $\beta < 0$
    \begin{align*}
        \int_s^t E_{\beta}(r)^2 dr + \int_0^{\infty}(E_{\beta}(t-s+r) - E_{\beta}(r))^2 dr \leq C(t-s)^{2\gamma}.
    \end{align*}
    Using \eqref{eq: VCIR expectation form}, the boundedness of higher moments follows from 
    \begin{align*}
        \E_{b,\beta}[X_t^{p}] &\leq 2^{p-1}\left(\E_{b,\beta}[X_t]\right)^{p} + 2^{p-1}\sigma^p C_p \E_{b,\beta}\left[ \left( \int_0^t E_{\beta}(t-s)^2 X_s ds \right)^{p/2}\right]
        \\ &\leq 2^{p-1}\left(\E_{b,\beta}[X_t]\right)^p + 2^{p-1}\sigma^p C_p \left( \int_0^t E_{\beta}(s)^2ds\right)^{\frac{p}{2}-1} \int_0^t E_{\beta}(t-s)^2\E_{b,\beta}[X_s^{p/2}] ds
        \\ &\leq 2^{p-1}\left(\sup_{t \geq 0}\E_{b,\beta}[X_t]\right)^p + 2^{p-1}\sigma^p C_p \left( \int_0^t E_{\beta}(s)^2ds\right)^{\frac{p}{2}} \sup_{t \geq 0}\E_{b,\beta}[X_t^{p/2}].
    \end{align*}
    and iteration applied to $p = 2^n$ with $n \geq 2$. For the increments of the process, we first estimate
    \begin{align*}
        \E_{b,\beta}[|X_t - X_s|^p] &\leq 2^{p-1}\left| \E_{b,\beta}[X_t] - \E_{b,\beta}[X_s]\right|^p 
        \\ &\qquad + 2^{p-1}\sigma^p C_p \E_{b,\beta}\left[ \left( \int_0^s (E_{\beta}(t-r) - E_{\beta}(s-r))^2 X_r dr\right)^{p/2} \right]
        \\ &\qquad + 2^{p-1}\sigma^p C_p \E_{b,\beta}\left[ \left( \int_s^t E_{\beta}(t-r)^2 X_r dr\right)^{p/2} \right].
    \end{align*}
    The first term can be estimated by the explicit form of the firm moment combined with \eqref{eq: Ebeta tail bound}. For the third term, we obtain
    \begin{align*}
        \E_{b,\beta}\left[ \left( \int_s^t E_{\beta}(t-r)^2 X_r dr\right)^{p/2} \right]
        &\leq \left( \int_s^t E_{\beta}(t-r)^2 dr \right)^{\frac{p}{2}-1} \int_s^t E_{\beta}(t-r)^2 \E_{b,\beta}[X_r^{p/2}] dr
        \\ &\leq \sup_{t \geq 0} \sup_{(b,\beta) \in \Theta}\E_{b,\beta}[X_t^{p/2}] \left( \int_s^t E_{\beta}(t-r)^2 dr \right)^{\frac{p}{2}}
    \end{align*}
    which gives the desired bound. The second term can be estimated in the same way.     
\end{proof}

\subsection{Asymptotic independence}

In this section, we prove that the VCIR process and its stationary law are asymptotically independent. It follows from the variation of constants formula for Volterra equations that $V(\cdot;\mu)$ given by \eqref{eq: V mu} is also the unique solution of the Volterra Riccatti equation
\begin{align}\label{eq: V mu Ebeta}
 V(t;\mu) = \int_{[0,t]}E_{\beta}(t-s)\mu(ds) - \frac{\sigma^2}{2} \int_0^t E_{\beta}(t-s)V(s;\mu)^2ds.
\end{align}
As a first step, we prove bounds on the integrals of $V(\cdot;\mu)$.

\begin{Lemma}\label{lemma: V mu Ebeta}
 Suppose that $\mu(\R_+) < \infty$, (K1) and (K2) hold and that $\beta < 0$. Then  
 \[
  \int_T^{S}V(t;\mu)^p dt \leq \mu(\R_+)^{p-1}\left( \int_{[0,T]} \int_{T-s}^{S-s}E_{\beta}(t)^pdt\mu(ds) + \int_{(T,S)}\int_0^{S-s}E_{\beta}(t)^pdt \mu(ds)\right)
 \]
 holds for $p = 1,2$ and $0 \leq T < S \leq +\infty$.
\end{Lemma}
\begin{proof}
    Noting that $V(\cdot; \mu) \geq 0$, in view of \eqref{eq: V mu Ebeta} we arrive at
    \begin{align}\label{eq: Vmu upper bound}
        0 \leq V(t;\mu) \leq \int_{[0,t]}E_{\beta}(t-s)\mu(ds),\qquad t > 0.
    \end{align}
    Using the Jensen inequality 
    $\left(\int_{[0,t]}E_{\beta}(t-s)\mu(ds)\right)^p \leq \mu(\R_+)^{p-1}\int_{[0,t]}E_{\beta}(t-s)^p \mu(ds)$ and then the Fubini theorem, we arrive at 
    \begin{align*}
        \int_T^{S} V(t;\mu)^p dt &\leq \int_T^{S}\left( \int_{[0,t]}E_{\beta}(t-s)\mu(ds)\right)^p dt
        \\ &\leq \mu(\R_+)^{p-1}\int_T^{S}\int_{[0,t]}E_{\beta}(t-s)^p \mu(ds)dt
        \\ &= \mu(\R_+)^{p-1}\bigg( \int_{[0,T]} \int_T^{S}E_{\beta}(t-s)^pdt \mu(ds) 
        \\ &\qquad \qquad \qquad \qquad \qquad \qquad  + \int_{(T,S)}\int_s^{S}E_{\beta}(t-s)^pdt \mu(ds)\bigg)
        \\ &= \mu(\R_+)^{p-1}\left( \int_{[0,T]} \int_{T-s}^{S-s}E_{\beta}(t)^pdt\mu(ds) + \int_{(T,S)}\int_0^{S-s}E_{\beta}(t)^pdt \mu(ds)\right). 
    \end{align*}
    This proves the assertion.
\end{proof}

Below we prove the asymptotic independence property \eqref{eq: asymptotic independence}.

\begin{Theorem}\label{thm: asymptotic independence}
 Suppose that conditions (K1) and (K2) are satisfied and that $\beta < 0$ holds. Then for each fixed $n \geq 1$, 
 \[
    \P_{b,\beta} \circ (X_{t_1}, \dots, X_{t_n})^{-1} \Longrightarrow \pi_{b,\beta}^{\otimes n} \ \text{ and } \ \P_{b,\beta}^{\mathrm{stat}} \circ (X_{t_1},\dots, X_{t_n})^{-1} \Longrightarrow \pi_{b,\beta}^{\otimes n}
 \]
 holds weakly locally uniformly in $(b,\beta) \in \R_+ \times (-\infty, 0)$ as $t_1,\dots, t_n \to \infty$ such that $\min_{k=1,\dots, n}|t_{k+1} - t_k| \to \infty$ and $0 < t_1 < \dots < t_n$.
\end{Theorem}
\begin{proof}
    Without loss of generality, assume that $\Theta = [0, b^*] \times [\beta_*, \beta^*] \subset \R_+ \times (-\infty, 0)$ for some $b^* \geq 0$ and $\beta_* < \beta^* < 0$. We prove that the Laplace transform converges uniformly on $\Theta$. The assertion is then a consequence of Theorem \ref{thm: uniform Laplace transform convergence}.
    
    \textit{Step 1.} Let us first prove the assertion for $\P_{b,\beta}$. Let $V = V(\cdot;\mu)$ be the unique solution of \eqref{eq: V mu} with $\mu$ given as in \eqref{eq: mu VCIR finite dimensional}. Using the affine transformation formula \eqref{eq: VCIR Laplace} we obtain
    \begin{align*}
    \E_{b,\beta}\left[ e^{- \sum_{k=1}^n u_k X_{t_k}}\right]
    &= \E_{b,\beta}\left[ e^{- \int_{[0,t_n]}X_{t_n-s}\mu(ds)}\right]
   \\ &= \exp\left( -x_0 \sum_{k=1}^n u_k - x_0\int_0^{t_n} R(V(s))ds - b \int_0^{t_n} V(s)ds\right)
    \\ &= \prod_{k=0}^{n-1} \exp\left( - x_0 u_k - x_0\int_{t_n - t_{k+1}}^{t_n - t_{k}} R(V(s))ds - b \int_{t_n-t_{k+1}}^{t_{n}-t_{k}} V(s)ds\right)
    \\ &= \prod_{k=0}^{n-1} \exp\left( - x_0 u_k - x_0\int_{0}^{t_{k+1} - t_{k}}R(V_k(s))ds - b \int_0^{t_{k+1} - t_k} V_k(s)ds\right)
    \end{align*}
    where we have set $t_0 = 0$ and $V_k(s) = V(s + (t_n - t_{k+1}))$. Let $\overline{V}(\cdot;u)$ be the solution of \eqref{eq: V mu} with $\mu = u \delta_0$, $u \geq 0$. According to \eqref{eq: pi affine formula}, it suffices to show that, for each $k=0,\dots, n-1$,
    \begin{align}\label{eq: V convergence}
     \int_{0}^{t_{k+1} - t_{k}} (x_0R(V_k(s)) + bV_k(s))ds \longrightarrow \int_0^{\infty}\left( x_0R(\overline{V}(s;u_{k+1})) + b\overline{V}(s;u_{k+1})\right)ds
    \end{align}
    holds uniformly with respect to $(b,\beta) \in \Theta$. Note that, an application of Lemma \ref{lemma: V mu Ebeta} to $\overline{V}$ and $S = + \infty$ gives for each $T \geq 0$
    \begin{align}\label{eq: V hat L2 bound}
     \int_T^{\infty} \left(\overline{V}(\tau; u) + \overline{V}(\tau; u)^2 \right)d\tau \leq u \int_T^{\infty}E_{\beta^*}(\tau)d\tau + u^2 \int_T^{\infty}E_{\beta^*}(\tau)^2 d\tau < \infty
    \end{align}
    where we have used that $E_{\beta} \leq E_{\beta^*}$, see \eqref{eq: Ebeta monotonicity}. Since $|R(x)| \leq |\beta|x + \frac{\sigma^2}{2}x^2$ for $x \geq 0$, the right-hand side in \eqref{eq: V convergence} is well defined. The property \eqref{eq: V convergence} will be shown in the subsequent steps 2 -- 4, while the assertion for $\P_{b,\beta}^{\mathrm{stat}}$ will be shown in step 5.

    \textit{Step 2.} Recall that $\mu$ is given as in \eqref{eq: mu VCIR finite dimensional}, whence $\mu(\R_+) = \sum_{j=1}^n u_j$. Using Lemma \ref{lemma: V mu Ebeta} with $T = 0$ gives for $k=0,\dots, n-1$ and $p = 1,2$
    \begin{align*}
        \int_0^{\infty}V_k(\tau)^p d\tau 
        &= \int_{t_n - t_{k+1}}^{\infty} V(\tau)^p d\tau
        \\ &\leq \int_0^{\infty} V(\tau)^p d\tau 
        \leq \left( \sum_{j=1}^{n}u_j \right)^p \int_0^{\infty}E_{\beta}(t)^p dt.
    \end{align*}
    Moreover, if $0 < T < t_{k+1} - t_{k}$, we obtain again from Lemma \ref{lemma: V mu Ebeta} applied to $S = t_{n} - t_k$
    \begin{align*}
     \int_T^{t_{k+1} - t_k} V_k(\tau)^p d\tau
     &= \int_{T+t_n - t_{k+1}}^{t_{n} - t_k} V(\tau)^p d\tau 
     \\ &\leq \left( \sum_{j=1}^n u_j\right)^{p-1} \int_{[0,T + t_n - t_{k+1}]} \int_{T + t_n - t_{k+1}-s}^{t_{n} - t_k - s}E_{\beta}(t)^pdt\mu(ds) 
     \\ &\qquad + \left( \sum_{j=1}^n u_j\right)^{p-1}\int_{(T + t_n - t_{k+1},t_n - t_k)}\int_0^{t_n - t_{k}-s}E_{\beta}(t)^pdt \mu(ds).
    \end{align*}
    
    Using the particular form of $\mu$, we see that $\mu((T+t_n - t_{k+1}, t_n - t_k))  \neq 0$ if and only if $T + t_n - t_{k+1} < t_n - t_l$ and $t_n - t_l < t_n - t_k$ holds for some $l = 1,\dots, n$. The second condition gives $l > k$ while the first condition is equivalent to $0 < T < t_{k+1} - t_l$ and hence is never satisfied for $l > k$. Hence $\mu((T+t_n - t_{k+1}, t_n - t_k)) = 0$ and the second term vanishes. For the first term, we obtain
    \begin{align*}
        &\ \int_{[0,T + t_n - t_{k+1}]} \int_{T + t_n - t_{k+1}-s}^{t_{n} - t_k - s}E_{\beta}(t)^pdt\mu(ds)
        \\ &= \sum_{l=1}^n u_l \1_{ \{ t_n - t_l \leq T + t_n - t_{k+1}\} }\int_{T + (t_l - t_{k+1})}^{t_l - t_k}E_{\beta}(\tau)^pd\tau
        \\ &= \sum_{l=k+1}^n u_l \int_{T + (t_l - t_{k+1})}^{t_l - t_k}E_{\beta}(\tau)^pd\tau
        \leq \left(\sum_{j=1}^{n}u_j \right)\int_T^{\infty}E_{\beta}(\tau)^pd\tau
    \end{align*}
    where we have used that for $l \leq k$ one has $T + t_l - t_{k+1} \leq T - (t_{k+1}-t_k) < 0$. To summarise, we have shown that
    \begin{align}\label{eq: Vk L2 bound}
    \int_T^{t_{k+1} - t_k} V_k(\tau)^p d\tau \leq \left( \sum_{j=1}^n u_j \right)^{p}\int_{T}^{\infty}E_{\beta}(\tau)^p d\tau
    \leq \left( \sum_{j=1}^n u_j \right)^{p}\int_{T}^{\infty}E_{\beta^*}(\tau)^p d\tau
    \end{align}
    holds for $p = 1,2$ and each $0 \leq T < t_{k+1} - t_k$.

    \textit{Step 3.} In this step, we prove that, for each $T > 0$ and $k = 0, \dots, n-1$, we have
    \begin{align}\label{eq: Vk convergence}
     \sup_{(b,\beta) \in \Theta}\int_0^T |V_k(\tau) - \overline{V}(\tau;u_{k+1})|^2 d\tau \longrightarrow 0
    \end{align}
    as $t_1,\dots, t_n \to \infty$ with $|t_{k+1} - t_k| \to \infty$.

    Fix $T > 0$, $k = 0, \dots, n-1$ and let $\tau \in [0,T] \subset [0, t_{k+1}-t_k)$. Then $\tau + (t_{n} - t_{k+1}) \in [t_n - t_{k+1}, t_n - t_k)$ and hence the function $V_k$ satisfies according to \eqref{eq: V mu Ebeta}
    \begin{align}\label{eq: equation for Vk}
     V_k(\tau) &= \sum_{j=k+1}^{n}u_j E_{\beta}(\tau + (t_{n} - t_{k+1}) - (t_n -  t_j)) 
     \\ &\qquad - \frac{\sigma^2}{2} \int_0^{\tau + (t_{n}-t_{k+1})} E_{\beta}(\tau + (t_{n}-t_{k+1}) - s)V(s)^2 ds \notag
     \\ &= \sum_{j=k+1}^{n}u_j E_{\beta}(\tau - (t_{k+1} - t_j)) 
     - \frac{\sigma^2}{2} \int_{-(t_n - t_{k+1})}^{\tau} E_{\beta}(\tau - s)V_k(s)^2 ds \notag
     \\ &= u_{k+1}E_{\beta}(\tau) - \frac{\sigma^2}{2}\int_0^{\tau}E_{\beta}(\tau - s)V_k(s)^2 ds + A_{k}(\tau) \notag
    \end{align}
    where we have set
    \[
     A_k(\tau) = \sum_{j=k+2}^{n}u_j E_{\beta}(\tau - (t_{k+1} - t_j)) - \frac{\sigma^2}{2}\int_0^{t_n - t_{k+1}} E_{\beta}(\tau+s)V(t_n - t_{k+1} - s)^2ds.
    \]
    
    Using \eqref{eq: V hat L2 bound} combined with the first bound from step 2, we find that 
    \begin{align}\label{eq: L2 bound}
     \| V_k + \overline{V}(\cdot; u_{k+1}) \|_{L^2(\R_+)}^2
     &\leq 2 \| V_k \|_{L^2(\R_+)}^2 + 2 \|\overline{V}(\cdot; u_{k+1})\|_{L^2(\R_+)}^2
     \\ &\leq 2 \left( \sum_{j=1}^{n}u_j\right)^2 \|E_{\beta^*}\|_{L^2(\R_+)}^2 + 2 u_{k+1}^2 \|E_{\beta^*}\|_{L^2(\R_+)}^2 =: C. \notag
    \end{align}
    Using \eqref{eq: equation for Vk} combined with representation \eqref{eq: V mu Ebeta} applied to $\overline{V}(\cdot; u_{k+1})$, and then the Cauchy-Schwartz inequality we arrive at
    \begin{align*}
       &\ |V_k(\tau) - \overline{V}(\tau;u_{k+1})|^2 
       \\ &\leq 2A_k(\tau)^2 + 2 \left(\frac{\sigma^2}{2}\right)^2\left(\int_0^{\tau}E_{\beta^*}(\tau-s)(V_k(s)^2 - \overline{V}(s;u_{k+1})^2)ds \right)^2
        \\ &\leq 2A_k(\tau)^2 + \frac{\sigma^4C}{2} \int_0^{\tau}E_{\beta^*}(\tau-s)^2|V_k(s) - \overline{V}(s;u_{k+1})|^2ds.
    \end{align*}
    Set $\rho(t) = \frac{\sigma^4C}{2}E_{\beta^*}(t)^2$ and let $\varkappa \in L_{loc}^1(\R_+)$ be the unique solution of $\varkappa(t) = \rho(t) + \int_0^t \rho(t-s)\varkappa(s)ds$. Using \cite[Proposition 9.8.1]{MR1050319} we find that $\varkappa \geq 0$ and using a Volterra analogue of the Gronwall inequality (see e.g. \cite[Theorem A.2]{MR4019885}), gives
    \[
    |V_k(\tau) - \overline{V}(\tau;u_{k+1})|^2 \leq 2A_k(\tau)^2 +  2\int_0^{\tau}\varkappa(\tau - r)A_k(r)^2dr.
    \] 
    Integrating over $[0,T]$ and then using Young's inequality in $L^1([0,T])$ gives
    \begin{align*}
    \int_0^T |V_k(\tau) - \overline{V}(\tau; u_{k+1})|^2 d\tau 
    &\leq 2\int_0^T |A_k(\tau)|^2 d\tau + 2\int_0^T \int_0^{\tau}\varkappa(\tau - r)A_k(r)^2dr d\tau 
    \\ &\leq 2\left( 1 + \int_0^T \varkappa(\tau)d\tau\right)\int_0^T |A_k(\tau)|^2 d\tau.
    \end{align*}
    Thus, \eqref{eq: Vk convergence} is proven, once we have shown that $\sup_{(b,\beta) \in \Theta}\int_0^T |A_k(\tau)|^2 d\tau \longrightarrow 0$. For this purpose, we first note that 
    \begin{align*}
    \int_0^T |A_k(\tau)|^2 d\tau &\leq 2n\sum_{j=k+2}^n u_j^2 \int_0^T E_{\beta^*}(\tau + (t_j - t_{k+1}))^2 d\tau 
    \\ &\qquad + \frac{\sigma^4}{2}\int_0^T \left( \int_0^{t_n - t_{k+1}}E_{\beta^*}(\tau+s)V(t_n - t_{k+1}-s)^2ds \right)^2 d\tau
    \end{align*}
    where we have used $E_{\beta} \leq E_{\beta^*}$. The first term converges to zero by dominated convergence and
    \[
     \int_0^T E_{\beta^*}(\tau + (t_j - t_{k+1}))^2 d\tau
     = \int_{t_j - t_{k+1}}^{T + t_j - t_{k+1}} E_{\beta^*}(\tau)^2 d\tau  
     \longrightarrow 0.
    \]
    The second term equals zero when $k + 1 = n$. Thus, let us assume that $k +1 < n$. Then using the substitution $t_n - t_{k+1} - s = r$ we obtain
    \begin{align*}
     &\ \int_0^T \left( \int_0^{t_n - t_{k+1}}E_{\beta^*}(\tau+s)V(t_n - t_{k+1}-s)^2ds \right)^2 d\tau
     \\ &= \int_0^T \left( \int_0^{t_n - t_{k+1}} E_{\beta^*}(\tau + t_n - t_{k+1} - r)V(r)^2dr \right)^2 d\tau
     \\ &\leq \int_0^T (E_{\beta^*} \ast V^2)(\tau + (t_n-t_{k+1}))^2 d\tau
     = \int_{t_n-t_{k+1}}^{T + t_n - t_{k+1}} (E_{\beta^*} \ast V^2)(\tau)^2 d\tau.
    \end{align*}
    Using \eqref{eq: Vmu upper bound} combined with $E_{\beta} \leq E_{\beta^*}$ we obtain 
    \begin{align}\label{eq: V square}
     V(t) \leq \left( \sum_{j=1}^n u_j\right) \int_{[0,t]} E_{\beta^*}(t-s) \mu(ds)
    \end{align}
    with the right-hand side being square-integrable in $t$. Hence
    \[
     \sup_{(b,\beta) \in \Theta}\int_{t_n-t_{k+1}}^{T + t_n - t_{k+1}} (E_{\beta^*} \ast V^2)(\tau)^2 d\tau \longrightarrow 0
    \]
    by dominated convergence, since $E_{\beta^*} \in L^2(\R_+)$ and $V^2$ has uniform majorant \eqref{eq: V square} in $L^1(\R_+)$. This proves \eqref{eq: Vk convergence} and hence completes the proof of step 3.

    \textit{Step 4.} Let us now show that \eqref{eq: Vk convergence} implies \eqref{eq: V convergence}. Indeed, take $T > 0$ and let $0 \leq t_1 < \dots < t_n$ be such that $|t_{k+1} - t_k| > T$ for all $k = 0, \dots, n-1$ with $t_0 = 0$. Taking the difference in \eqref{eq: V convergence}, we obtain 
    \begin{align*}
        &\ \left| \int_0^{\infty}(x_0 R(\overline{V}(s;u_{k+1})) + b \overline{V}(s,u_{k+1}))ds - \int_0^{t_{k+1} - t_k}(x_0 R(V_k(s)) + bV_k(s))ds \right|
        \\ &\leq x_0 \int_0^T |R(\overline{V}(s;u_{k+1})) - R(V_k(s))|ds
         + b \int_0^T |\overline{V}(s;u_{k+1}) - V_k(s)|ds
         \\ &\qquad + \int_T^{t_{k+1}- t_k} \left( x_0 |R(V_k(s))| + bV_k(s) \right) ds
        \\ &\qquad + \int_{T}^{\infty} \left( x_0|R(\overline{V}(s;u_{k+1}))| + b \overline{V}(s;u_{k+1}) \right)ds
        \\ &\leq (|\beta|x_0 + b) \int_0^T |\overline{V}(s;u_{k+1}) - V_k(s)|ds + \frac{\sigma^2\sqrt{C}x_0}{2}\| \overline{V}(\cdot;u_{k+1}) - V_k\|_{L^2([0,T])}
        \\ &\qquad + (|\beta|x_0 + b)\int_T^{t_{k+1}-t_k}V_k(s)ds + \frac{\sigma^2x_0}{2}\int_T^{t_{k+1}-t_k}V_k(s)^2 ds
        \\ &\qquad + (|\beta|x_0 + b) \int_T^{\infty} \overline{V}(s;u_{k+1}) ds + \frac{\sigma^2x_0}{2}\int_T^{\infty} \overline{V}(s;u_{k+1})^2 ds
    \end{align*}
    where we have used $|R(x) - R(y)| \leq |\beta||x-y| + \frac{\sigma^2}{2}(x+y)|x-y|$ and
    $|R(x)| \leq |\beta|x + \frac{\sigma^2}{2}x^2$ for $x,y \geq 0$, and also \eqref{eq: L2 bound} to bound $\| \overline{V}(\cdot;u_{k+1}) + V_k\|_{L^2([0,T])} \leq \sqrt{C}$. Because of \eqref{eq: Vk convergence}, the first two terms tend to zero. The third and fourth terms converge by \eqref{eq: Vk L2 bound} to zero as $T \to \infty$. Finally, the remaining two terms tend to zero due to \eqref{eq: V hat L2 bound} as $T \to \infty$. This proves the assertion for $X$.

    \textit{Step 5.} Finally, we prove the same assertion for the stationary process. Using the same $\mu$ and $V = V(\cdot; \mu)$, we find that
    \begin{align*}
        \E_{b,\beta}^{\mathrm{stat}}\left[ e^{- \sum_{k=1}^n u_k X_{t_k}} \right]
        &= \exp\left( - x_0 \sum_{k=1}^{n}u_k - x_0\int_0^{\infty}R(V(s))ds - b \int_0^{\infty}V(s)ds \right)
        \\ &= \prod_{k=0}^{n-1} \exp\left( - x_0 u_k - x_0\int_{0}^{t_{k+1} - t_{k}} R(V_k(s))ds - b \int_0^{t_{k+1} - t_k} V_k(s)ds\right)
        \\ &\qquad \qquad \cdot \exp\left( - x_0\int_{t_n}^{\infty}R(V(s))ds - b \int_{t_n}^{\infty}V(s)ds \right).
    \end{align*} 
    In steps 2 -- 4, we have already seen that 
    \begin{align*}
     &\ \prod_{k=0}^{n-1} \exp\left( - x_0 u_k - x_0\int_{0}^{t_{k+1} - t_{k}} R(V_k(s)) ds - b \int_0^{t_{k+1} - t_k} V_k(s)ds\right)
     \\ &\qquad \longrightarrow \prod_{k=0}^{n-1} \exp\left( - x_0 u_k - x_0\int_{0}^{\infty} (R(\overline{V}(s;u_{k+1}))ds - b \int_0^{\infty} \overline{V}(s;u_{k+1})ds \right)
    \end{align*}
    holds uniformly on $\Theta$. Using the particular form of $R$, the assertion follows for $p=1,2$ from
    \begin{align*}
        \int_{t_n}^{\infty} V(s)^p ds 
        &\leq \left(\sum_{j=1}^n u_j\right)^{p-1}\int_{[0,t_n]}\int_{t_n-s}^{\infty}E_{\beta^*}(t)^p dt \mu(ds)
        \\ &\leq \left(\sum_{j=1}^n u_j\right)^{p-1} \sum_{l=1}^n u_l \int_{t_j}^{\infty}E_{\beta^*}(t)^p dt
        \leq \left(\sum_{j=1}^n u_j\right)^{p} \int_{t_1}^{\infty}E_{\beta^*}(t)^p dt
    \end{align*}
    where we have used Lemma \ref{lemma: V mu Ebeta} and that $\mu((t_n,\infty)) = 0$.
\end{proof}

\subsection{Law-of-large numbers and ergodicity}

In this section, we first derive the Law of Large Numbers as stated in \eqref{eq: Birkhoff VOU} for $X$ and the stationary process $X^{\mathrm{stat}}$. The following is our main result.

\begin{Theorem}\label{thm: law of large numbers VCIR}
 Suppose that conditions (K1) and (K2) are satisfied, $\beta < 0$ and that $\sigma > 0$. Then the following assertions hold:
 \begin{enumerate}
     \item[(a)] If $f: [0, \infty) \longrightarrow \R$ is continuous and polynomially bounded, then \eqref{eq: Birkhoff VOU} holds in $L^p(\Omega, \P_{b,\beta})$ for each $p \in [2,\infty)$ and all $(b,\beta) \in \R_+ \times (-\infty, 0)$. If $f$ is additionally uniformly continuous, then this convergence also holds locally uniformly in $(b,\beta)$.
          
     \item[(b)] Suppose that $f: [0,\infty) \longrightarrow \R$ is Lebesgue almost everywhere continuous on $(0,\infty)$, has polynomial growth, and there exists $0 < \lambda \leq \frac{3\gamma}{2}$ such that 
     \begin{align}\label{eq: K lower bound}
        \inf_{h \in [0,1]}h^{-\lambda}\int_0^h K(t)^2 dt > 0.
    \end{align}
    If $f$ is continuous in zero or $\pi_{b,\beta}(\{0\}) = 0$, then \eqref{eq: Birkhoff VOU} holds in $L^p(\Omega, \P_{b,\beta})$ for each $p \in [2,\infty)$ and all $(b,\beta) \in \R_+ \times (-\infty, 0)$. 

     \item[(c)] Suppose there exists $0 < \lambda \leq \frac{3\gamma}{2}$ such that \eqref{eq: K lower bound} holds. Let $\Theta \subset \R_+ \times (-\infty, 0)$ be a compact such that 
     \begin{align}\label{eq: pi boundary}
        \lim_{\eta \searrow 0}\sup_{(b,\beta) \in \Theta}\pi_{b,\beta}(\{|x| \leq \eta\}) = 0.
     \end{align}
     Let $f: [0,\infty) \longrightarrow \R$ be Lebesgue almost everywhere continuous on $(0,\infty)$ satisfying for some $p \in [1,\infty)$ 
     \[
        \lim_{R \to \infty}\sup_{(b,\beta) \in \Theta}\left(\sup_{t \geq 0}\ \E_{b,\beta}\left[|f(X_t)|^p \1_{\{|f(X_t)| > R\}}\right] + \int_{\R_+} |f(x)| \1_{\{|f(x)|> R\}}\pi_{b,\beta}(dx)\right) = 0.
     \] 
     Then \eqref{eq: Birkhoff VOU} holds in $L^p(\Omega, \P_{b,\beta})$ for this particular choice of $p$ uniformly in $(b,\beta) \in \Theta$. 
 \end{enumerate} 
\end{Theorem}
\begin{proof}
 The assertion is a consequence of the Birkhoff Lemma \cite[Lemma 2.1]{BFK24} applied to $Y_t = f(X_t)$. Indeed, it follows from the asymptotic independence that $X_t \Longrightarrow \pi_{b,\beta}$ and $(X_t,X_s) \Longrightarrow \pi_{b,\beta} \otimes \pi_{b,\beta}$ locally uniformly in $(b,\beta)$. In case (a), we may apply the continuous mapping theorem for fixed $(b,\beta)$ to conclude that $Y$ satisfies the assumptions of the Birkhoff Lemma \cite[Lemma 2.1]{BFK24}. When $f$ is also uniformly continuous, the same argument can be used with the uniform continuous mapping Theorem \ref{thm: uniform CMT} instead. For case (b), the assertion follows by the usual continuous mapping theorem with $(b,\beta)$ fixed, provided that $\pi_{b,\beta}(D) = 0$ where $D$ denotes the set of all continuity points of $f$. The latter property is satisfied since by \cite[Section 6]{FJ22} the limit distribution $\pi_{b,\beta}$ is when restricted onto $(0,\infty)$ absolutely continuous with respect to the Lebesgue measure. 
 
 Finally, in case of condition (c), we may apply Theorem \ref{thm: uniform CMT}.(c) provided that the density of the limit distribution $\pi_{b,\beta}$ satisfies condition \eqref{eq: locally absolute continuity} with respect to the Lebesgue measure. To prove the latter, let us first observe that \eqref{eq: pi boundary} yields $\pi_{b,\beta}(\{0\}) = 0$ for $(b,\beta) \in \Theta$, and hence $\pi_{b,\beta}$ is absolutely continuous with respect to the Lebesgue measure on $\R_+$, see \cite[Section 6]{FJ22}. Thus, writing $\pi_{b,\beta}(x)dx = \pi_{b,\beta}(dx)$ and noting that $\pi_{b,\beta}(x) = 0$ for $x < 0$, the same reference gives $\sqrt{|\cdot|}\pi_{b,\beta} \in B_{1,\infty}^{\delta}(\R)$ for some $\delta \in (0,1)$ small enough. By using Proposition \ref{prop: VCIR technical stuff}, it is clear that the estimates given in \cite[Section 6]{FJ22} can be obtained uniformly on $\Theta$, which gives 
 $\sup_{(b,\beta) \in \Theta}\|\sqrt{|\cdot|}\pi_{b,\beta}\|_{B_{1,\infty}^{\delta}(\R)} < \infty$. Let $\e \in (0,1)$ and take $\eta > 0$ such that $\sup_{(b,\beta) \in \Theta}\pi_{b,\beta}(\{|x| \leq \eta\}) < \e/2$. Let $A \subset \R_+$ be a Borel set with $\lambda(A) < \infty$. Then we obtain $\pi_{b,\beta}(A) \leq \e/2 + \int_{\eta}^{\infty}\1_A(x) \pi_{b,\beta}(x)dx$. To bound the latter integral, we use the H\"older inequality with $p \in (1,1/(1-\delta))$ and then the continuous embedding $B_{1,\infty}^{\lambda}(\R) \hookrightarrow L^p(\R)$ to find 
 \begin{align*}
     \int_{\eta}^{\infty}\1_A(x) \pi_{b,\beta}(x)dx 
     &\leq \eta^{-1/2}\int_{\eta}^{\infty}\1_A(x) \sqrt{x}\pi_{b,\beta}(x)dx
     \\ &\leq \eta^{-1/2} \left(\int_{\eta}^{\infty}\1_A(x)dx \right)^{1 - \frac{1}{p}} \| \sqrt{|\cdot|}\pi_{b,\beta}\|_{L^p(\R)}
     \\ &\leq C(p, \delta)\eta^{-1/2} \lambda(A)^{1 - \frac{1}{p}} \sup_{(b,\beta) \in \Theta}\| \sqrt{|\cdot|}\pi_{b,\beta}\|_{B_{1,\infty}^{\delta}(\R)}.
 \end{align*}
 Thus we may choose $\lambda(A)$ small enough such that $\sup_{(b,\beta) \in \Theta}\pi_{b,\beta}(A) < \e$ which implies \eqref{eq: locally absolute continuity} and hence proves $f(X_t) \Longrightarrow \pi_{b,\beta}\circ f^{-1}$ uniformly on $\Theta$. Similarly, we can show that $(f(X_t), f(X_s)) \Longrightarrow (\pi_{b,\beta}\otimes \pi_{b,\beta})\circ (f,f)^{-1}$. 
 
 An application of \cite[Lemma 2.1]{BFK24} to $f_R(x) = f(x) \wedge R$ implies that \eqref{eq: Birkhoff VOU} holds for each $p \geq 0$. Let $\e > 0$. Then we obtain for $R > 0$ 
 \begin{align*}
     &\ \sup_{(b,\beta) \in \Theta}\left\| \frac{1}{t}\int_0^t f(X_s)ds - \int_{\R_+}f(x)\pi_{b,\beta}(dx)\right\|_{L^p(\Omega, \P_{b,\beta})}
     \\ &\leq \sup_{(b,\beta) \in \Theta} \frac{1}{t}\int_0^t \left\| f(X_s) \1_{\{|f(X_s)|>R\}} \right\|_{L^p(\Omega,\P_{b,\beta})} ds
     \\ &\qquad + \sup_{(b,\beta) \in \Theta}\left\| \frac{1}{t}\int_0^t f_R(X_s)ds - \int_{\R_+}f_R(x)\pi_{b,\beta}(dx)\right\|_{L^p(\Omega, \P_{b,\beta})}
     \\ &\qquad + \sup_{(b,\beta) \in \Theta}\int_{\R_+} |f(x)| \1_{\{|f(x)|> R\}}\pi_{b,\beta}(dx)
     \\ &\leq \e + \sup_{(b,\beta) \in \Theta}\left\| \frac{1}{t}\int_0^t f_R(X_s)ds - \int_{\R_+}f_R(x)\pi_{b,\beta}(dx)\right\|_{L^p(\Omega, \P_{b,\beta})}
 \end{align*}
 where, by assumption, we may choose $R$ large enough independent of $t$ such that the first and third term are bounded by $\e/2$. Thus letting first $t \to \infty$ and then $\e \searrow 0$ proves the assertion.  
\end{proof}

Since merely distributional properties have been used in the above proof, which are also satisfied by the stationary process with law $\P_{b,\beta}^{\mathrm{stat}}$, all assertions of Theorem \ref{thm: law of large numbers VCIR} also hold for the stationary process.

Finally, we show that the VCIR process with stationary law $\P_{b,\beta}^{\mathrm{stat}}$ is ergodic. The latter results complement the existing literature where analogous results are usually obtained for Markov processes in terms of convergence of transition probabilities in the total variation distance, see e.g. \cite{BDF23, MR4567723, MR4585666, MR3343292, MR4102262} for related results on classical affine processes.  

\begin{Corollary}\label{cor: VCIR mixing}
 Suppose that conditions (K1) and (K2) are satisfied, $\beta < 0$ and that $\sigma > 0$. Then the stationary VCIR process $X^{\mathrm{stat}}$ with law $\P_{b,\beta}^{\mathrm{stat}}$ is ergodic.
\end{Corollary}
\begin{proof}
 An application of the classical Birkhoff theorem to $\P_{b,\beta}^{\mathrm{stat}}$ implies that for each $f \in L^p(\R, \pi_{b,\beta})$, one has $\frac{1}{t}\int_0^t f(X_s)ds \longrightarrow \E_{b,\beta}^{\mathrm{stat}}[f(X_0)\ | \ \mathcal{I}]$ a.s. with respect to $\P_{b,\beta}^{\mathrm{stat}}$ where $\mathcal{I}$ is the shift-invariant sigma-algebra. Let $A \in \mathcal{I}$. Using the stationarity of the process and the shift-invariance of $A$, we obtain $\E_{b,\beta}^{\mathrm{stat}}[ f(X_0)\1_{A}] = \E_{b,\beta}^{\mathrm{stat}}[ f(X_h)\1_{A}]$ and hence $\E_{b,\beta}^{\mathrm{stat}}[f(X_0)\ | \ \mathcal{I}] = \E_{b,\beta}^{\mathrm{stat}}[f(X_h)\ | \ \mathcal{I}]$ for any $h \geq 0$. Using the previously shown law of large numbers and Birkhoff's ergodic theorem, the uniqueness of the limits implies that $\int_{\R}f(x) \pi_{b,\beta}(dx) = \E_{b,\beta}^{\mathrm{stat}}[f(X_h) \ | \ \mathcal{I}]$ holds $\P_{b,\beta}^{\mathrm{stat}}\text{-a.s.}$ for each $h \geq 0$. Since this relation holds for all $f: \R \longrightarrow \R$ continuous and bounded, we obtain $\pi_{b,\beta}(A) = \P_{b,\beta}^{\mathrm{stat}}[X_h \in A \ | \ \mathcal{I}]$ for all Borel sets $A$. Hence we obtain for a Borel set $A \subset \R$ and $B \in \mathcal{I}$
 \begin{align*}
  \E_{b,\beta}^{\mathrm{stat}}\left[ \1_{\{X_h \in A\}} \1_B \right]
  &= \E_{b,\beta}^{\mathrm{stat}}\left[ \P_{b,\beta}^{\mathrm{stat}}\left[ X_h \in A \ | \ \mathcal{I}\right] \1_B\right]
  \\ &= \pi_{b,\beta}(A) \P_{b,\beta}^{\mathrm{stat}}[B]
  \\ &= \P_{b,\beta}^{\mathrm{stat}}[X_0 \in A]\P_{b,\beta}^{\mathrm{stat}}[B]
  = \P_{b,\beta}^{\mathrm{stat}}[X_h \in A]\P_{b,\beta}^{\mathrm{stat}}[B].
 \end{align*}
 Hence for each $h \geq 0$, $X_h$ is independent of $\mathcal{I}$ with respect to $\P_{b,\beta}^{\mathrm{stat}}$, and thus $\vee_{h \geq 0}\sigma(X_h)$ is independent of $\mathcal{I}$. Since $\mathcal{I} \subset \vee_{h \geq 0}\sigma(X_h)$, we conclude that $\mathcal{I}$ is independent of itself and hence trivial. This proves that $\P_{b,\beta}^{\mathrm{stat}}$ is ergodic.
\end{proof}

\subsection{Method of moments for fractional kernel}

In this part, we apply the law of large numbers to prove consistency for the estimators obtained from the method of moments. To allow for explicit computations, below we consider the case of the fractional kernel \eqref{eq: fractional kernel}. Define $m_1(T), m_2(T)$ by
\[
 m_1(T) = \frac{1}{T}\int_0^T X_t\tau(dt)\ \text{ and } \
 m_2(T) = \frac{1}{T}\int_0^T X_t^2 \tau(dt)
\]
where $\tau(dt)$ is either the counting measure on $\N_0$ or the Lebesgue measure on $\R_+$. Then $m_1(T) \longrightarrow m_1(b,\beta)$ and $m_2(T) \longrightarrow m_2(b,\beta)$ by Proposition \ref{prop: VCIR technical stuff} with 
\begin{align*}
    m_1(b,\beta) = \frac{b}{|\beta|} \ \ \text{ and } \ \ m_2(b,\beta) = \left(\frac{b}{|\beta|}\right)^2 + b \sigma^2 C_{\alpha} |\beta|^{\frac{1}{\alpha} - 3}
\end{align*}
where we have used that $\int_0^{\infty}E_{\beta}(t)^2dt 
= C_{\alpha}|\beta|^{\frac{1}{\alpha} - 2}$ with the constant 
\[
 C_{\alpha} = \frac{1}{\pi}\int_0^{\infty} \frac{du}{1 + 2u^{\alpha}\cos(\frac{\pi \alpha}{2}) + u^{2\alpha}},
\]
which follows from the explicit formula $E_{\beta}(t) = t^{\alpha - 1}E_{\alpha,\alpha}(\beta t^{\alpha})$ with $E_{\alpha,\alpha}$ the two-parameter Mittag-Leffler function.
Solving these equations gives
 \begin{align*}
     \widehat{b}_T &= m_1(T) \left( \frac{C_{\alpha} \sigma^2 m_1(T)}{m_2(T) - m_1(T)^2} \right)^{\frac{\alpha}{2\alpha - 1}}
     \\ \widehat{\beta}_T &= -\left( \frac{C_{\alpha} \sigma^2 m_1(T)}{m_2(T) - m_1(T)^2} \right)^{\frac{\alpha}{2\alpha - 1}}.
 \end{align*}

The following result applies equally to continuous and discrete low-frequency observations.

\begin{Corollary}\label{cor: VCIR method of moments}
    Suppose that $K$ is given by the fractional kernel \eqref{eq: fractional kernel}, $\sigma > 0$ and $\beta < 0$. Then $(\widehat{b}_T, \widehat{\beta}_T)$ is consistent in probability locally uniform in $(b,\beta)$, i.e.
    \[
     \lim_{T \to \infty}\sup_{(b,\beta) \in \Theta}\P\left[ \left| (\widehat{b}_T, \widehat{\beta}_T) - (b,\beta)\right| > \e\right] = 0
    \]
    holds for each $\e > 0$ and any compact $\Theta \subset \R \times (-\infty, 0)$. Furthermore, if $(T_n)_{n \geq 1}$ is a sequence with $T_n \nearrow \infty$ and there exists $p \in [1,\infty)$ with 
    \[
     \sum_{n=1}^{\infty}\left| 1 - \frac{T_n}{T_{n+1}}\right|^p < \infty,
    \]
    then $(\widehat{b}_{T_n}, \widehat{\beta}_{T_n})$ is strongly consistent.
\end{Corollary}
\begin{proof}
 The uniform Law of Large Numbers Theorem \ref{thm: law of large numbers VCIR} combined with the uniform continuous mapping theorem Proposition \ref{thm: uniform CMT}.(a) gives the consistency of the estimators locally uniformly in $(b,\beta)$. An application of \cite[Lemma 2.5]{BFK24} implies the desired strong consistency of the estimators. 
\end{proof}

\section{Application to Maximum-likelihood estimation}

\subsection{Continuous-time observations}

In this section, we derive the MLE for the VCIR process and prove the asymptotic consistency and normality. Suppose that conditions (K1) and (K2) are satisfied, and let $X$ be the VCIR process with parameters $(b,\beta)$ and law $\P_{b,\beta}$. Unless the Volterra kernel $K$ is regular, $X$ is not a semimartingale. Thus, we cannot use classical results on the change of measures to define the maximum-likelihood ratio and derive the corresponding MLE $(\widehat{b}_T, \widehat{\beta}_T)$ for the parameter $(b,\beta)$. To overcome this issue, we follow the method recently developed in \cite{BFK24}. 

Denote by $L$ the resolvent of the first kind of $K$, i.e., a locally finite measure $L$ on $\R_+$ such that $\int_{[0,t]}K(t-s)L(ds) = 1$ for $t \geq 0$. Note that under condition (K1) such a measure exists and satisfies $L(ds) = K(0_+)\delta_0(ds) + L_0(s)ds$ with $L_0 \in L_{loc}^1(\R_+)$ being completely monotone (see \cite[Chapter 5, Theorem 5.5]{MR1050319}). Suppose that there exists $\alpha \in (1/2,1)$ such that
\[
 \sup_{t \in (0,T]}\left( t^{1-\alpha}K(t) + t^{2-\alpha}|K'(t)| \right) < \infty, \qquad T > 0.
\]
Denote by $C([0,T])$ the space of continuous functions on $[0,T]$, by $C^{\eta}([0,T])$ the space of $\eta \in [0,1)$ H\"older continuous functions, and by $C_0^{\eta}([0,T])$ its subspace of functions vanishing in $0$. The functional 
\begin{align}\label{eq: Z process}
 Z_t(x) := \int_{[0,t]}(x_{t-s} - x_0)L(ds), \qquad x \in C([0,T])
\end{align}
defines a continuous mapping $Z: C^{\eta}([0,T]) \longrightarrow C_0^{\eta}([0,T])$ for $\eta \in [0,1]$. Its inverse transformation $\Gamma: C_0^{\eta}([0,T]) \longrightarrow C_0([0,T])$ is given by 
\[
 \Gamma_t(z) = K(t)z_0 + \int_0^t K'(t-s)(z_s - z_t)ds
\]
and satisfies $\Gamma(Z(x)) = x$ and $Z(\Gamma(z)) = z$ for $x,z \in C_0^{\eta}([0,T])$ whenever $\alpha + \eta > 1$, see \cite[Lemma 2.1]{BFK24}.

Inserting $X$ into the functional $Z$, it is easy to see that $Z(X)$ is a $\P_{b,\beta}$-semimartingale with differential
\begin{align}\label{eq: Z VCIR semimartingale}
 dZ_t(X) = (b + \beta X_t)dt + \sigma \sqrt{X_t}dB_t.
\end{align} 
Let $\P_{b,\beta}^Z$ be the law of $Z$ under $\P_{b,\beta}$. The process $Z$ solves the path-dependent diffusion equation
\[
 dZ_t = (b + \beta (\Gamma_t(Z) + x_0))dt + \sigma \sqrt{\Gamma_t(Z) + x_0}dB_t.
\]
A formal application\footnote{ignoring the assumption (I) and (7.123) therein} of \cite[Theorem 7.19]{MR1800857} shows that the log-likelihood ratio (provided it exists) would be given by 
\begin{align*}
    \log\ \frac{d \P_{b,\beta}^Z}{d\P_{b_0,\beta_0}^Z}\bigg|_{\F_T} = \ell_T(b,\beta; b_0, \beta_0; Z)
\end{align*}
where $\ell_T(b,\beta; b_0, \beta_0; Z)$ is given by
\begin{align*}
  &\int_0^T \frac{(b - b_0) + (\beta - \beta_0)(\Gamma_t(Z)+ x_0)}{\sigma^2 (\Gamma_t(Z) + x_0)}dZ_t 
 \\ &\qquad +  \frac{1}{2}\int_0^T \frac{\left[(b_0 + b) + (\beta_0+\beta)(\Gamma_t(Z)+x_0)\right]\left[ (b_0 - b) + (\beta_0 - \beta)(\Gamma_t(Z) + x_0)\right]}{\sigma^2 (\Gamma_t(Z) + x_0)}dt.
\end{align*}
A rigorous proof that the probability measures $\P^Z_{b,\beta}$ and $\P^Z_{b_0,\beta_0}$ are locally equivalent is related to the boundary behaviour of the process, and shall be studied in future work. The interested reader may consult \cite{BP24} for some related results. However, the parameter estimators defined below turn out to be consistent and asymptotically normal regardless of the absolute continuity of $\P_{b,\beta}^Z|_{\mathcal{F}_T}$ and $\P_{b_0, \beta_0}^Z|_{\mathcal{F}_T}$.

Performing the substitution $Z = Z(X)$ and setting $b_0 = \beta_0 = 0$, we obtain 
\begin{align*}
  \ell_T(b,\beta; Z) &= \int_0^T \frac{b + \beta X_t}{\sigma^2 X_t}dZ_t(X) - \frac{1}{2}\int_0^T \frac{\left(b + \beta X_t\right)^2}{\sigma^2 X_t}dt.
\end{align*}
Assuming that $\int_0^T \E_{b,\beta}[X_t^{-1}]dt < \infty$ a.s., it follows that
\[
 M_T = \int_0^T \begin{pmatrix} X_t^{-1/2} \\ X_t^{1/2} \end{pmatrix}dB_t \ \text{ with } \langle M\rangle_T = \begin{pmatrix}
        \int_0^T X_t^{-1}dt  & T \\ T & \int_0^TX_t dt
        \end{pmatrix}
\]
defines a square-integrable martingale, and we obtain
\[
 \sigma^2 \nabla_{(b,\beta)}\ell_T(b,\beta; X) = - \langle M \rangle_T \begin{pmatrix} b \\ \beta \end{pmatrix} + \begin{pmatrix} \int_0^T X_t^{-1}dZ_t(X) \\ Z_T(X) \end{pmatrix}.
\]
Since the determinant $D_T := \mathrm{det}(\langle M \rangle_T) = \int_0^T X_t dt \int_0^T X_t^{-1}dt - T^2$ is a.s. strictly positive by the Cauchy-Schwarz inequality and the fact that $X$ is non-constant due to $\sigma > 0$ and $K \neq 0$, we may solve $\sigma^2 \nabla_{(b,\beta)}\ell_T(b,\beta; X) = 0$. The solution is given by 
\begin{align}\label{eq: MLE estimators}
    \begin{pmatrix} \widehat{b}_T \\ \widehat{\beta}_T \end{pmatrix} = \langle M \rangle_T^{-1} \begin{pmatrix} \int_0^T X_t^{-1}dZ_t(X) \\ Z_T(X) \end{pmatrix} \ \text{ with } \ \langle M\rangle_T^{-1} = \frac{1}{D_T}\begin{pmatrix}
         \int_0^TX_t dt & -T \\ -T & \int_0^T X_t^{-1}dt
        \end{pmatrix}.
\end{align}
By explicit calculation, we find the desired expressions for the MLE, namely
\begin{align*}
        \widehat{b}_T(X) &= \frac{\int_0^T X_t dt \int_0^T \frac{1}{X_t}dZ_t(X) - T Z_T(X)}{\int_0^T X_tdt \int_0^T \frac{1}{X_t}dt - T^2},
        \\ \widehat{\beta}_T(X) &= \frac{Z_T(X) \int_0^T \frac{1}{X_t}dt - T \int_0^T \frac{1}{X_t}dZ_t(X)}{\int_0^T X_tdt \int_0^T \frac{1}{X_t}dt - T^2}.
\end{align*} 

Below, we prove that these estimators are strongly consistent and asymptotically normal.

\begin{Theorem}\label{thm: MLE for VCIR}
 Suppose that conditions (K1) and (K2) are satisfied, $\beta < 0$, $\sigma > 0$, and \eqref{eq: K lower bound} holds for some $\lambda \in (0,3\gamma/2)$. Let $\Theta \subset \R_+ \times (-\infty, 0)$ be a compact such that
 \begin{align}\label{eq: VCIR boundary}
    \exists \e > 0: \qquad \sup_{(b,\beta) \in \Theta} \sup_{t \geq 0}\ \E_{b,\beta}\left[ X_t^{-1 - \e} \right] < \infty.
 \end{align}
 Then $(\widehat{b}_T, \widehat{\beta}_T)$ is strongly consistent, satisfies 
 \begin{align}\label{eq: consistency}
     \lim_{T \to \infty}\sup_{(b,\beta) \in \Theta}\P_{b,\beta}\left[ \left| (\widehat{b}_T, \widehat{\beta}_T) - (b,\beta)\right| > \eta \right] = 0, \qquad \forall \eta > 0,
 \end{align}
  and is asymptotically normal in the sense that
 \begin{align}\label{eq: asymptotic normality}
     (\mathbb{P}_{b,\beta})_*\left(\sqrt{T}(\widehat{b}_T - b, \widehat{\beta}_T - \beta) \right) \Longrightarrow \sigma I(b,\beta)^{-1/2}\mathcal{N}(0,\mathrm{id}_{2 \times 2}), \qquad T \longrightarrow \infty
 \end{align}
 holds uniformly in $(b,\beta) \in \Theta$ with Fisher information matrix
 \[
    I(b,\beta) = \begin{pmatrix} \int_{\R_+}\frac{1}{x}\pi_{b,\beta}(dx) & 1 \\ 1 & \int_{\R_+}x \pi_{b,\beta}(dx) \end{pmatrix}.
 \]
\end{Theorem}
\begin{proof}
    Using \eqref{eq: Z VCIR semimartingale} combined with \eqref{eq: MLE estimators}, we obtain by direct computation
    \begin{align*}
        \begin{pmatrix} \widehat{b}_T \\ \widehat{\beta}_T \end{pmatrix} &= \begin{pmatrix}
            b \\ \beta
        \end{pmatrix} - \sigma \langle M \rangle_T^{-1} M_T.
    \end{align*}
    Note that condition \eqref{eq: pi boundary} is satisfied by the Fatou's Lemma combined with 
    \begin{align*}
        \sup_{(b,\beta) \in \Theta}\pi_{b,\beta}(\{|x| \leq \eta\}) &\leq \eta^{1+\e} \sup_{(b,\beta) \in \Theta}\int_{\R_+} x^{-1 -\e}\pi_{b,\beta}(dx)
        \\ &\leq \eta^{1+\e} \sup_{(b,\beta) \in \Theta} \sup_{t \geq 0}\E_{b,\beta}\left[ X_t^{-1-\e} \right] < \infty.
    \end{align*}
    Likewise, it holds that 
    \begin{align*}
        &\ \sup_{(b,\beta) \in \Theta}\int_{\R_+} x^{-1} \1_{\{x^{-1}> R\}}\pi_{b,\beta}(dx)
        + \sup_{(b,\beta) \in \Theta}\sup_{t \geq 0}\ \E_{b,\beta}\left[ X_t^{-1} \1_{\{X_t^{-1} > R\}}\right]
        \\ &\qquad \qquad \qquad \leq \frac{2}{R^{\e}}\sup_{(b,\beta) \in \Theta} \sup_{t \geq 0}\E_{b,\beta}\left[ X_t^{-1 - \e} \right] < \infty.
    \end{align*}
    Thus Theorem \ref{thm: law of large numbers VCIR}.(c) can be applied to $f(x) = x$ and $f(x) = 1/x$ with $p=1$ from which we deduce $\langle M \rangle_T/T \longrightarrow I(b,\beta)$ in probability uniformly on $\Theta$ as $T \to \infty$. In particular, $\langle M \rangle_T \longrightarrow +\infty$ a.s. and hence, by the strong law of large numbers for martingales, we conclude that $\langle M \rangle_T^{-1}M_T \longrightarrow 0$ a.s., which proves strong consistency. Similarly, we also obtain $\langle M \rangle_T^{-1}/T \longrightarrow I(b,\beta)^{-1}$ and $M_T/T \longrightarrow 0$ in probability uniformly on $\Theta$. Hence, an application of Theorem \ref{thm: uniform continuous mapping theorem}.(a) yields
    \begin{align*}
    \sup_{(b,\beta) \in \Theta}\P[ |(\widehat{b}_T, \widehat{\beta}_T) - (b,\beta)| > \e]
        &= \sup_{(b,\beta) \in \Theta}\P\left[ \left| \left(\frac{\langle M_T\rangle_T}{T}\right)^{-1} \frac{M_T}{T}\right| > \sigma^{-1}\e\right] \longrightarrow 0.
    \end{align*}
    The asymptotic normality uniformly in $(b,\beta) \in \Theta$ follows from the central limit theorem \cite[Proposition 1.21]{MR2144185} applied to
    \[
     \left( \frac{\langle M \rangle_T}{T}\right)^{-1}\frac{M_T}{\sqrt{T}} \Longrightarrow I(b,\beta)^{-1} \mathcal{N}(0, I(b,\beta)) = I(b,\beta)^{-1/2}\mathcal{N}(0,\mathrm{id}_{2\times 2}).
    \] 
\end{proof}

While $\int_{\R_+}x \pi_{b,\beta}(dx) = \lim_{t \to \infty}\E_{b,\beta}[X_t]$ can be computed explicitly (see Section 2), an explicit formula for $\int_{\R_+}x^{-1}\pi_{b,\beta}(dx)$ is not known. Such an expression could be computed, e.g., by Monte-Carlo methods based on the Law of Large Numbers established in Theorem \ref{thm: law of large numbers VCIR}.

\subsection{Discrete high-frequency observations}

In this section, we consider the case of discrete high-frequency observations. For each $n \geq 1$ let $0 = t_0^{(n)} < \dots < t_n^{(n)}$ be a partition of length $t_n^{(n)}$. Denote by
\begin{align}\label{eq: partition}
    \mathcal{P}_n = \{ [t_k^{(n)}, t_{k+1}^{(n)}] \ : \ k = 0,\dots, n-1 \}
\end{align}
the collection of neighbouring intervals determined by this partition, and let
\[
    |\mathcal{P}_n| = \max_{[u,v] \in \mathcal{P}_n}(v-u)
\]
be its mesh size. The discretisation of $X$ with respect to $\mathcal{P}_n$ is defined by
\begin{align}\label{eq: X discretisation}
    X^{\mathcal{P}_n}_s = \sum_{[u,v] \in \mathcal{P}_n}\1_{[u,v)}(s)X_{u}.
\end{align} 
Here and below we let $\lesssim$ denote an inequality that is supposed to hold up to a constant independent of the discretisation and locally uniform in $(b,\beta) \in \R_+ \times (-\infty, 0)$. 

Then, under conditions (K1), (K2), $\sigma > 0$, and $\beta < 0$, and using Proposition \ref{prop: VCIR technical stuff}, it is easy to see that 
\begin{align}\label{eq: discretisation bound 1}
        \sup_{n \geq 1}\sup_{t \in [0,t_n^{(n)})}\left(\|X_t\|_{L^{p}(\Omega)} + \|X^{\mathcal{P}_n}_t\|_{L^{p}(\Omega)}\right) \lesssim C_p < \infty, \qquad \forall p \in [1,\infty)
\end{align}
and 
\begin{align}\label{eq: discretisation bound 2}
    \| X_t - X^{\mathcal{P}_n}_t\|_{L^p(\Omega)} \lesssim |\mathcal{P}_n|^{\gamma}, \qquad t \in [0,t_n).
\end{align}

Finally, let us recall that $Z = Z(X)$ denotes the process given by \eqref{eq: Z process} where $x$ is replaced with $X$. Both bounds combined allow us to prove the following lemma.

\begin{Lemma}\label{lemma: discretisation VOU}
 Suppose that (K1), (K2) hold, $\sigma > 0$ and $\beta < 0$. Let $f: \R \longrightarrow \R$ satisfy
 \[
  |f(x) - f(y)| \leq C_f(1 + |x|^q + |y|^q)|x-y|, \qquad \forall x,y \in \R,
 \]
 for some constants $C_f,q\geq0$. Then for each $p \in [2,\infty)$ there exists $C_p(b,\beta)$ independent of the discretisation and locally bounded with respect to $(b,\beta) \in \R \times (-\infty, 0)$ such that
 \begin{align*}
      &\ \left\| \int_0^{t_n^{(n)}} f(X_s)ds -  \sum_{[u,v] \in \mathcal{P}_n}f(X_{u})(v-u) \right\|_{L^p(\Omega)} 
      \\ &\quad + \left\| \int_0^{t_n^{(n)}} f(X_s)dZ_s - \sum_{[u,v]\in \mathcal{P}_n}f(X_u)(Z_v - Z_u) \right\|_{L^p(\Omega)} 
     \leq C_p(b,\beta) \left( \left(t_n^{(n)}\right)^{\frac{1}{2}} + t_n^{(n)} \right)|\mathcal{P}_n|^{\gamma}.
 \end{align*}
\end{Lemma}
\begin{proof}
 Using $\sum_{[u,v] \in \mathcal{P}_n}f(X_{u})(v-u) = \int_0^{t_n^{(n)}}f(X^{\mathcal{P}_n}_s)ds$ combined with the H\"older inequality, and properties \eqref{eq: discretisation bound 1} and \eqref{eq: discretisation bound 2}, the first asserted inequality follows from 
    \begin{align*}
        &\ \left\| \int_0^{t_n^{(n)}} f(X_s)ds - \sum_{[u,v] \in \mathcal{P}_n}f(X_{u})(v-u) \right\|_{L^p(\Omega)}
        \\ &\lesssim \int_{0}^{t_{n}^{(n)}} \left\|(1 + |X_s|^q + |X^{\mathcal{P}_n}_{s}|^q)|X_s - X^{\mathcal{P}_n}_{s}| \right\|_{L^p(\Omega)} ds 
        \\ &\lesssim \int_{0}^{t_{n}^{(n)}} \left\| X_s - X^{\mathcal{P}_n}_{s} \right\|_{L^{2p}(\Omega)} ds 
       \lesssim |\mathcal{P}_n|^{\gamma} t_n^{(n)}.
    \end{align*}
    For the second bound, we first use $\sum_{[u,v] \in \mathcal{P}_n}f(X_{u})(Z_{v} - Z_{u})
     = \int_0^{t_n^{(n)}}f(X^{\mathcal{P}_n}_s)dZ_s$ and then the semimartingale representation \eqref{eq: Z VCIR semimartingale} to find 
    \begin{align*}
    &\ \left\| \int_0^{t_n^{(n)}} f(X_s)dZ_s - \sum_{[u,v] \in \mathcal{P}_n}f(X_{u})(Z_{v} - Z_{u})\right\|_{L^p(\Omega)} 
    \\ &\leq \int_{0}^{t_n^{(n)}} \left\| (f(X_s) - f(X^{\mathcal{P}_n}_{s}))(b+\beta X_s) \right\|_{L^p(\Omega)} ds 
    \\ &\qquad + \sigma \left\|  \int_{0}^{t_{n}^{(n)}} (f(X_s) - f(X^{\mathcal{P}_n}_{s}))\sqrt{X_s}dB_s \right\|_{L^p(\Omega)} 
    \\ &\lesssim \int_{0}^{t_{n}^{(n)}} \left\| (1 + |X_s|^q + |X^{\mathcal{P}_n}_{s}|^q)(|b|+|\beta|| X_s|)|X_s - X^{\mathcal{P}_n}_{s}| \right\|_{L^{p}(\Omega)} ds 
    \\ &\qquad + \left( \E\left[ \left(\int_{0}^{t_{n}^{(n)}} (1 + |X_s|^q + |X^{\mathcal{P}_n}_{s}|^q)^2|X_s - X^{\mathcal{P}_n}_{s}|^2 X_s ds\right)^{\frac{p}{2}} \right] \right)^{\frac{1}{p}}
    \\ &\lesssim \int_{0}^{t_{n}^{(n)}} \left\| X_s - X^{\mathcal{P}_n}_{s}\right\|_{L^{2p}(\Omega)} ds 
    + \left( t_n^{(n)}\right)^{\frac{1}{2} - \frac{1}{p}}\left( \int_{0}^{t_{n}^{(n)}} \left( \E\left[|X_s - X^{\mathcal{P}_n}_{s}|^{2p} \right] \right)^{1/2}ds \right)^{\frac{1}{p}}
    \\ &\lesssim \left(t_n^{(n)} + \left( t_n^{(n)}\right)^{\frac{1}{2}} \right)|\mathcal{P}_n|^{\gamma},
    \end{align*}
    where we have used the H\"older inequality combined with \eqref{eq: discretisation bound 1} and \eqref{eq: discretisation bound 2}. 
\end{proof}

Below we complement these estimates with additional approximations of the integrals $\int_0^T X_t^{-1}dt$ and $\int_0^T X_t^{-1}dZ_t(X)$.

\begin{Lemma}\label{lemma: discretisation VCIR}
    Suppose that conditions (K1), (K2) hold, $\sigma > 0$ and $\beta < 0$. Suppose that $\Theta \subset \R_+ \times (-\infty, 0)$ is a compact such that 
    \begin{align}\label{eq: VCIR negative moment bound}
        \exists \varepsilon > 0: \qquad \sup_{(b,\beta) \in \Theta} \sup_{t \geq 0}\ \E_{b,\beta}\left[ X_t^{-3-\e} \right] < \infty
    \end{align}
    holds. Then   
    \begin{align*}
     &\ \sup_{(b,\beta) \in \Theta}\E_{b,\beta}\left[\left| \int_0^{t_n^{(n)}}\frac{1}{X_s}ds - \sum_{[u,v] \in \mathcal{P}_n}\frac{v-u}{X_u}\right| \right]
     \\ &\qquad + \sup_{(b,\beta) \in \Theta}\E_{b,\beta}\left[\left|\int_0^{t_n^{(n)}}  \frac{1}{X_s}dZ_s(X)  -  \sum_{[u,v] \in \mathcal{P}_n} \frac{Z_v(X) - Z_u(X)}{X_u^{\mathcal{P}_n}}\right|\right]
     \\ &\qquad \qquad \lesssim \left( t_n^{(n)} + (t_n^{(n)})^{1/2}\right)|\mathcal{P}_n|^{\gamma}.
    \end{align*}
\end{Lemma}
\begin{proof}
    First note that by \eqref{eq: VCIR negative moment bound} and the particular form of $X^{\mathcal{P}_n}$ we have
    \begin{align}\label{eq: negative moments}
        \E_{b,\beta}\left[ (X_s^{\mathcal{P}_n})^{-3 - \e} \right] \leq \sup_{(b,\beta) \in \Theta} \sup_{t \geq 0}\ \E_{b,\beta}\left[ X_t^{-3 - \e} \right] < \infty
    \end{align}
    for each $s \in [0, t_n^{(n)})$. Hence, for the first bound we obtain from \eqref{eq: discretisation bound 2}
    \begin{align*}
       &\ \E_{b,\beta}\left[\left| \int_0^{t_n^{(n)}}\frac{1}{X_s}ds - \sum_{[u,v] \in \mathcal{P}_n}\frac{v-u}{X_u}\right| \right]
       \\ &= \E_{b,\beta}\left[\left|\int_0^{t_n^{(n)}} \frac{1}{X_s}ds  -  \int_0^{t_n^{(n)}} \frac{1}{X_s^{\mathcal{P}_n}}ds\right|\right]
        \\ &\leq \int_0^{t_n^{(n)}} \E_{b,\beta}\left[ \frac{|X_s - X_s^{\mathcal{P}_n}|}{X_s X_s^{\mathcal{P}_n}}\right]dx
        \\ &\lesssim |\mathcal{P}_n|^{\gamma}\int_0^{t_n^{(n)}} \left(\E_{b,\beta}\left[ X_s^{-q}\right]\right)^{1/q}\left(\E_{b,\beta}\left[ (X_s^{\mathcal{P}_n})^{-q}\right]\right)^{1/q}ds
        \\ &\lesssim t_n^{(n)}|\mathcal{P}_n|^{\gamma}
    \end{align*}
    uniformly on $\Theta$ where we have used the H\"older inequality for products of three functions with $q := q_1 = q_2 = 2+\varepsilon$ and $p = 1 + 2/\varepsilon$ determined by $1/p + 1/q_1 + 1/q_2 = 1$, and \eqref{eq: discretisation bound 1} and \eqref{eq: discretisation bound 2}. Similarly, we estimate 
    \begin{align*}
        &\ \E_{b,\beta}\left[\left|\int_0^{t_n^{(n)}}  \frac{1}{X_s}dZ_s(X)  -  \sum_{[u,v] \in \mathcal{P}_n} \frac{Z_v(X) - Z_u(X)}{X_u^{\mathcal{P}_n}}\right|\right] 
        \\ &\leq \E_{b,\beta}\left[\left|\int_0^{t_n^{(n)}}  \left(\frac{1}{X_s} - \frac{1}{X_s^{\mathcal{P}_n}}\right)(b+\beta X_s)ds \right|\right]
        \\ &\qquad + \sigma \E_{b,\beta}\left[\left|\int_0^{t_n^{(n)}}  \left(\frac{1}{X_s} - \frac{1}{X_s^{\mathcal{P}_n}}\right)\sqrt{X_s}dB_s \right|\right]\\ &\leq \int_0^{t_n^{(n)}} \E_{b,\beta}\left[ \frac{\left| X_s - X_s^{\mathcal{P}_n}\right| |b+\beta X_s|}{X_s X_s^{\mathcal{P}_n}} \right] ds
        + \left( \int_0^{t_n^{(n)}} \E_{b,\beta}\left[ \frac{|X_s - X_s^{\mathcal{P}_n}|^2}{X_s (X_s^{\mathcal{P}_n})^2} \right] ds \right)^{1/2}.
    \end{align*}
    For the first term we obtain from the H\"older inequality for products of three functions with $q := q_1 = q_2 = 2+\varepsilon$ and $p = 1 + 2/\varepsilon$, and then another application of the H\"older inequality with $r = \frac{3+\varepsilon}{2+\varepsilon}$ and $r' = 3+ \varepsilon$ the bound
    \begin{align*}
       &\ \int_0^{t_n^{(n)}} \E_{b,\beta}\left[ \left| X_s - X_s^{\mathcal{P}_n}\right| \frac{ |b+\beta X_s|^{1/2}}{X_s} \frac{ |b+\beta X_s|^{1/2}}{X_s^{\mathcal{P}_n}}  \right] ds 
        \\ &\lesssim \int_0^{t_n^{(n)}} \| X_s - X_s^{\mathcal{P}_n}\|_{L^p(\Omega)} \left( \E\left[ \frac{|b+\beta X_s|^{q/2}}{X_s^q}\right] \right)^{1/q} \left( \E\left[ \frac{|b+\beta X_s|^{q/2}}{(X_s^{\mathcal{P}_n})^q}\right] \right)^{1/q} ds 
        \\ &\lesssim |\mathcal{P}_n|^{\gamma} \int_0^{t_n^{(n)}} \left( \E\left[ |b + \beta X_s|^{\frac{(2+\varepsilon)(3+\varepsilon)}{2}} \right] \right)^{\frac{2}{(2+\varepsilon)(3+\varepsilon)}} \left( \E\left[ X_s^{-3-\varepsilon}\right] \right)^{\frac{1}{3+\varepsilon}} \left( \E\left[ (X_s^{\mathcal{P}_n})^{-3-\varepsilon}\right] \right)^{\frac{1}{3+\varepsilon}} ds 
        \\ &\lesssim t_n^{(n)} |\mathcal{P}_n|^{\gamma},
    \end{align*}
    where we have also used \eqref{eq: discretisation bound 1}, \eqref{eq: discretisation bound 2}, \eqref{eq: VCIR negative moment bound}, and \eqref{eq: negative moments}. Similarly, using once the H\"older inequality with $\frac{1}{1+\e/3} + \frac{1}{\frac{3+\e}{\e}} = 1$, and another time with $\frac{1}{3} + \frac{1}{3/2} = 1$, and then again \eqref{eq: discretisation bound 2}, \eqref{eq: VCIR negative moment bound}, and \eqref{eq: negative moments}, gives \begin{align*}
       &\ \int_0^{t_n^{(n)}} \E_{b,\beta}\left[ \frac{|X_s - X_s^{\mathcal{P}_n}|^2}{X_s (X_s^{\mathcal{P}_n})^2} \right] ds
        \\ &\leq \int_0^{t_n^{(n)}} \left(\E_{b,\beta}\left[ |X_s - X_s^{\mathcal{P}_n}|^{2\frac{3+\e}{\e}} \right]\right)^{\frac{\e}{3+\e}} \left(\E_{b,\beta}\left[ X_s^{-1-\e/3} (X_s^{\mathcal{P}_n})^{-2(1+\e/3)} \right] \right)^{\frac{1}{1 + \e/3}} ds
        \\ &\lesssim |\mathcal{P}_n|^{2\gamma} \int_0^{t_n^{(n)}} \left( \E\left[ X_s^{-3 - \e}\right]\right)^{\frac{1}{3 + \e}} \left( \E\left[ (X_s^{\mathcal{P}_n})^{-3 - \e} \right] \right)^{\frac{2}{3+\e}} ds
        \\ &\lesssim t_n^{(n)} |\mathcal{P}_n|^{2\gamma}.
    \end{align*}
    This proves the assertion.
\end{proof}

To obtain the MLE for high-frequency observations, below we consider the discretisation of the continuous time MLE $(\widehat{b}_T, \widehat{\beta}_T)$. Let $(m(n))_{n \geq 1}$ be a sequence with $m(n) \to \infty$ such that $\mathcal{P}_{m(n)}$ is a refinement of $\mathcal{P}_n$, i.e. $\mathcal{P}_n \subset \mathcal{P}_{m(n)}$. The discretised estimators are then given by

\begin{footnotesize}
\begin{align*}
        \widehat{b}^{(n)}(X) &= \frac{\left(\sum_{[u,v] \in \mathcal{P}_n}X_u (v-u) \right) \left( \sum_{[u,v] \in \mathcal{P}_n} X_u^{-1}( Z_v^{\mathcal{P}_m} - Z_u^{\mathcal{P}_m}) \right) - t_n^{(n)} Z^{\mathcal{P}_n}_{t_n^{(n)}}}{\left(\sum_{[u,v] \in \mathcal{P}_n}X_u (v-u) \right)\left(\sum_{[u,v] \in \mathcal{P}_n}X_u^{-1} (v-u) \right) - (t_n^{(n)})^2}
        \\ \widehat{\beta}^{(n)}(X) &= \frac{Z^{\mathcal{P}_n}_{t_n^{(n)}} \left(\sum_{[u,v] \in \mathcal{P}_n}X_u (v-u) \right) - t_n^{(n)} \left( \sum_{[u,v] \in \mathcal{P}_n} X_u^{-1}( Z_v^{\mathcal{P}_m} - Z_u^{\mathcal{P}_m}) \right)}{\left(\sum_{[u,v] \in \mathcal{P}_n}X_u (v-u) \right)\left(\sum_{[u,v] \in \mathcal{P}_n}X_u^{-1} (v-u) \right) - (t_n^{(n)})^2}
\end{align*}
\end{footnotesize}
where $Z^{\mathcal{P}_m} = Z^{\mathcal{P}_{m(n)}}$ denotes the discretised auxiliary process with respect to $\mathcal{P}_{m(n)}$ defined by 
\begin{align}\label{eq: Z discretisation}
     Z^{\mathcal{P}_{m(n)}}_{u}(X) = \sum_{\bfrac{[u',v'] \in \mathcal{P}_{m(n)}}{v' \leq u}}(X_{u - u'} - x_0)L((u', v']) + K(0_+)^{-1}(X_{u} - x_0)
\end{align}
where $[u,v] \in \mathcal{P}_n$, $Z^{\mathcal{P}_{m(n)}}_0(X) = 0$, and we use the convention $1/+\infty = 0$. The next theorem provides, under additional conditions, the consistency and asymptotic normality of this estimator.

\begin{Theorem}\label{thm: discrete MLE for VCIR}
   Suppose that (K1) and (K2) holds, $\sigma > 0$ and $\beta < 0$, and there exists $0 < \lambda \leq \frac{3\gamma}{2}$ satisfying \eqref{eq: K lower bound}. Let $(\mathcal{P}_n)_{n \geq 1}$ be a sequence of partitions such that $t_n^{(n)} \longrightarrow \infty$ as $n \to \infty$, and  
 \begin{align}\label{eq: mesh size}
  \lim_{n \to \infty} \sqrt{t_n^{(n)}} |\mathcal{P}_n|^{\gamma} = 0.
 \end{align}
 Let $(m(n))_{n\geq 1}$ be such that $m(n) \geq n$, $\mathcal{P}_n \subset \mathcal{P}_{m(n)}$ satisfying
 \begin{align}\label{eq: mesh size 1}
    \lim_{n \to \infty} n \cdot \frac{L((0,t_n^{(n)}])}{\sqrt{t_n^{(n)}}}|\mathcal{P}_{m(n)}|^{\gamma} = 0.
 \end{align}
 Suppose that $\Theta \subset \R \times (- \infty, 0)$ is a compact with the property \eqref{eq: VCIR negative moment bound}. Then 
 \begin{align}\label{eq: discrete consistency}
  \lim_{n \to \infty}\sup_{(b,\beta) \in \Theta}\P_{b,\beta}\left[ \left|(\widehat{b}^{(n)}, \widehat{\beta}^{(n)}) - (b, \beta) \right| > \e \right] = 0,
 \end{align}
 and asymptotically normality holds uniformly in $(b,\beta) \in \Theta$, i.e.
 \begin{align}\label{eq: discrete asymptotic normality}
  (\mathbb{\P}_{b,\beta})_*\left(\sqrt{t_n^{(n)}}\left( (\widehat{b}^{(n)}, \widehat{\beta}^{(n)}) - (b,\beta)\right) \right) \Longrightarrow \sigma I(b,\beta)^{-1/2}\mathcal{N}(0, \mathrm{id}_{2\times 2}), \qquad n \to \infty.
 \end{align}
 holds uniformly on $\Theta$ with Fisher information matrix given as in Theorem \ref{thm: MLE for VCIR}.
\end{Theorem}
\begin{proof} 
    \textit{Step 1.} Let us define $A_T = ( T^{-1}\int_0^T X_t^{-1}dZ_t(X), T^{-1} Z_T(X))^{\top}$ and define
    \begin{align*}
     \mathcal{D}^{(n)} &= \left( \frac{1}{t_n^{(n)}}\sum_{[u,v] \in \mathcal{P}_n} X_u (v-u) \right)\left( \frac{1}{t_n^{(n)}} \sum_{[u,v] \in \mathcal{P}_n} X_u^{-1} (v-u) \right) - 1,
     \\ \mathcal{M}^{(n)} &= \begin{pmatrix} \frac{1}{t_n^{(n)}} \sum_{[u,v] \in \mathcal{P}_n} X_u^{-1}(v-u) & 1 \\ 1 & \frac{1}{t_n^{(n)}} \sum_{[u,v] \in \mathcal{P}_n} X_u (v-u) \end{pmatrix},
     \\ \mathcal{A}^{(n)} &= \frac{1}{t_n^{(n)}} \begin{pmatrix} \sum_{[u,v] \in \mathcal{P}_n} X_u^{-1}\left( Z_v^{\mathcal{P}_m} - Z_u^{\mathcal{P}_m}\right) \\ Z_{t_n^{(n)}}^{\mathcal{P}_n}\end{pmatrix}.
    \end{align*}
    Then we obtain $(\widehat{b}^{(n)}, \widehat{\beta}^{(n)})^{\top} = (\mathcal{D}^{(n)})^{-1}\mathcal{M}^{(n)} \mathcal{A}^{(n)}$ and hence
    \begin{align*}
     \sqrt{t_n^{(n)}}\left(\begin{pmatrix}\widehat{b}^{(n)} \\ \widehat{\beta}^{(n)} \end{pmatrix} - \begin{pmatrix}
        b \\ \beta \end{pmatrix} \right)
        &= \sqrt{t_n^{(n)}} (\mathcal{D}^{(n)})^{-1} \mathcal{M}^{(n)}\left( \mathcal{A}^{(n)} - A_{t_n^{(n)}} \right) 
        \\ &\qquad + \sqrt{t_n^{(n)}} \left( (\mathcal{D}^{(n)})^{-1}\mathcal{M}^{(n)} - \langle M \rangle_{t_n^{(n)}}^{-1} \right) A_{t_n^{(n)}}
        \\ &\qquad + \sqrt{t_n^{(n)}} \langle M \rangle_{t_n}^{-1} A_{t_n^{(n)}}.
    \end{align*}
    Since $t_n^{(n)} \longrightarrow \infty$, the last term converges by Theorem \ref{thm: MLE for VCIR} to the desired Gaussian law locally uniformly in $(b,\beta)$. Thus, by Proposition \ref{prop: uniform weak convergence}.(a) it suffices to prove that the first two terms converge to zero in probability locally uniformly in $(b,\beta)$. 
    
    \textit{Step 2.} In this step, we prove that 
    \begin{align}\label{eq: step 2}
        \sqrt{t_n^{(n)}} (\mathcal{D}^{(n)})^{-1} \mathcal{M}^{(n)}\left( \mathcal{A}^{(n)} - A_{t_n^{(n)}} \right) \longrightarrow 0
    \end{align}
    in probability locally uniformly in $(b,\beta)$. For this purpose, let us write 
    \begin{align*}
         \mathcal{D}^{(n)} &- \mathrm{det}(I(b,\beta)) 
         \\ &= \mathcal{D}^{(n)} - D_{t_n^{(n)}} + D_{t_n^{(n)}} - \mathrm{det}(I(b,\beta))
        \\ &= \left( \frac{1}{t_n^{(n)}} \sum_{[u,v] \in \mathcal{P}_n} X_u (v-u) - \frac{1}{t_n^{(n)}}\int_0^{t_n^{(n)}}X_sds \right)\left( \frac{1}{t_n^{(n)}}\sum_{[u,v] \in \mathcal{P}_n}X_u^{-1}(v-u) \right)
        \\ &\qquad + \left( \frac{1}{t_n^{(n)}} \int_0^{t_n^{(n)}} X_s ds \right)\left( \frac{1}{t_n^{(n)}}\sum_{[u,v] \in \mathcal{P}_n}X_u^{-1}(v-u) - \frac{1}{t_n^{(n)}}\int_0^{t_n^{(n)}}X_s^{-1}ds \right)
        \\ &\qquad + D_{t_n^{(n)}} - \mathrm{det}(I(b,\beta)).
    \end{align*}
    By the previously established Law of Large Numbers from Theorem \ref{thm: law of large numbers VCIR}, we get $D_{t_n^{(n)}} \longrightarrow \mathrm{det}(I(b,\beta)) > 0$ in probability locally uniformly in $(b,\beta)$. For the first two terms, we may apply Lemma \ref{lemma: discretisation VOU}, Lemma \ref{lemma: discretisation VCIR}, and assumption \eqref{eq: mesh size}, combined with the Law of Large Numbers, \eqref{eq: discretisation bound 1}, and Slutsky's theorem, to find that these converge to zero as $n \to \infty$ in probability locally uniformly in $(b,\beta)$. Hence $\mathcal{D}^{(n)} \longrightarrow \mathrm{det}(I(b,\beta)) > 0$ in probability locally uniformly in $(b,\beta)$.
    
    Similarly, we may also show that $\mathcal{M}^{(n)} \longrightarrow I(b,\beta)$ holds in probability locally uniformly. Indeed, it suffices to study the convergence of the diagonal entries for these matrices. In both cases, letting $j \in \{-1,1\}$, we obtain
    \begin{align*}
        &\ \frac{1}{t_n^{(n)}} \sum_{[u,v] \in \mathcal{P}_n} X_u^j (v-u) - \int_{\R_+} x^j \pi_{b,\beta}(x)dx 
        \\ &= \frac{1}{t_n^{(n)}} \sum_{[u,v] \in \mathcal{P}_n} X^j_u (v-u) - \frac{1}{t_n^{(n)}}\int_0^{t_n^{(n)}} X^j_s ds + \frac{1}{t_n^{(n)}}\int_0^{t_n^{(n)}} X_s^j ds - \int_{\R_+} x^j \pi_{b,\beta}(x)dx 
    \end{align*}
    which tends to zero in probability locally uniformly by the Law of Large Numbers combined with Lemma \ref{lemma: discretisation VOU} and \eqref{eq: mesh size}. 

    Next we prove 
    \begin{align}\label{vector discretisation}
        \sqrt{t_n^{(n)}}\left( \mathcal{A}^{(n)} - A_{t_n^{(n)}} \right) \longrightarrow 0
    \end{align}
    by studying the convergence for both components separately. First, let us note that by \eqref{eq: discretisation bound 2} we obtain
    \begin{align}\label{eq: Z discretisation 1}
     \left\| Z_{u} - Z^{\mathcal{P}_m}_u\right\|_{L^p(\Omega)}
     \leq \int_{[u,v]} \|X_{u-s} - X_{u-s}^{\mathcal{P}_m}\|_{L^p(\Omega)} L(ds)
     \lesssim |\mathcal{P}_m|^{\gamma} L((0,u])
    \end{align}
    for $[u,v] \in \mathcal{P}_n$ and $m \geq n$. Then we obtain for the second component
    \[
        \left\| \sqrt{t_n^{(n)}}\left(\frac{1}{t_n^{(n)}}Z_{t_n^{(n)}}^{\mathcal{P}_n} - \frac{1}{t_n^{(n)}} Z_{t_n^{(n)}} \right) \right\|_{L^p(\Omega)} \lesssim |\mathcal{P}_m|^{\gamma}\frac{L((0,t_n^{(n)}])}{\sqrt{t_n^{(n)}}}
    \]
    which tends to zero by \eqref{eq: mesh size 1}. For the first component, we use Lemma \ref{lemma: discretisation VCIR} combined with \eqref{eq: Z discretisation 1} to find 
    \begin{align*}
        &\ \sqrt{t_n^{(n)}} \left\| \frac{1}{t_n^{(n)}}\sum_{[u,v] \in \mathcal{P}_n} X_u^{-1}\left( Z_v^{\mathcal{P}_m} - Z_u^{\mathcal{P}_m}\right) - \frac{1}{t_n^{(n)}} \int_0^{t_n^{(n)}} X_t^{-1}dZ_t \right\|_{L^1(\Omega)}
        \\ &\lesssim \sqrt{t_n^{(n)}} \left\| \frac{1}{t_n^{(n)}}\sum_{[u,v] \in \mathcal{P}_n} X_u^{-1}\left( (Z_v^{\mathcal{P}_m} - Z_v) - (Z_u^{\mathcal{P}_m} - Z_u) \right) \right\|_{L^1(\Omega)}
        \\ &\qquad + \sqrt{t_n^{(n)}} \left\| \frac{1}{t_n^{(n)}}\sum_{[u,v] \in \mathcal{P}_n} X_u^{-1}\left( Z_v -  Z_u \right) - \frac{1}{t_n^{(n)}} \int_0^{t_n^{(n)}} X_t^{-1}dZ_t \right\|_{L^1(\Omega)}
        \\ &\lesssim \frac{|\mathcal{P}_m|^{\gamma}}{\sqrt{t_n^{(n)}}} \sum_{[u,v] \in \mathcal{P}_n} \|X_u^{-1}\|_{L^2(\Omega)}( L((0,u]) + L((0, v]) ) + (1 + \sqrt{t_n^{(n)}}) |\mathcal{P}_m|^{\gamma}
        \\ &\lesssim  |\mathcal{P}_m|^{\gamma} n \frac{L((0,t_n^{(n)}])}{\sqrt{t_n^{(n)}}} + \sqrt{t_n^{(n)}} |\mathcal{P}_m|^{\gamma}.
    \end{align*}
    By \eqref{eq: discretisation bound 1} and \eqref{eq: discretisation bound 2}, the right-hand side tends to zero. Thus, we have shown that \eqref{vector discretisation}. Thus, since $\mathcal{D}^{(n)} \longrightarrow \mathrm{det}(I(b,\beta))$, $\mathcal{M}^{(n)} \longrightarrow I(b,\beta)$, and \eqref{vector discretisation}, by Slutsky's theorem we obtain \eqref{eq: step 2}.

    \textit{Step 3.} Arguing similarly to step 2, we may also prove 
    \[
     \sqrt{t_n^{(n)}} \left( (\mathcal{D}^{(n)})^{-1}\mathcal{M}^{(n)} - \langle M \rangle_{t_n^{(n)}}^{-1} \right) \longrightarrow 0
    \]
    in probability locally uniformly in $(b,\beta)$.
    Since additionally
    \[
        A_{t_n^{(n)}} \longrightarrow \left( \int_{\R_+} \frac{b + \beta x}{x} \pi_{b,\beta}(x)dx, \int_{\R_+}(b + \beta x) \pi_{b,\beta}(x)dx \right)^{\top}.
    \]
    An application of Slutsky's theorem proves the assertion.  
\end{proof}

Condition \eqref{eq: VCIR boundary} and the stronger condition \eqref{eq: VCIR negative moment bound} are both variants of boundary non-attainment for the VCIR process. For the Markovian Cox-Ingersoll-Ross process (i.e. $K \equiv 1$), such conditions can be directly verified from the explicit form of the density of $\mathcal{L}_{b,\beta}(X_t)$ and $\pi_{b,\beta}$ provided that $\sigma$ and $b$ satisfy the so-called Feller condition $2b/\sigma^2 > 3$. A Volterra analogue of the Feller condition was recently obtained in \cite{BP24} for regular kernels. Non-regular kernels (e.g. the fractional kernel) have not yet been treated in the literature. Hence, the joint estimation of $(b,\beta)$ is yet conditionally given assumptions \eqref{eq: VCIR boundary} and \eqref{eq: VCIR negative moment bound}. The validity of these assumptions, i.e. a Volterra analogue of the Feller conditions, shall be studied in future research. It is interesting to note that such boundary attainment conditions may be also closely linked with the existence of strong solutions since the diffusion coefficient $x \longmapsto \sqrt{x}$ is smooth on $(0,\infty)$, and the Lipschitz continuity is only violated at $x = 0$.

Finally, we briefly comment on the estimation of parameters where either $b$ or $\beta$ is known. The proofs of the remarks below can be obtained in the same way (essentially simplified) as for the joint estimation of parameters.

\begin{Remark}\label{remark: VCIR estimation beta}
    Suppose that $b \geq 0$ is known, conditions (K1) and (K2) are satisfied, that $\beta < 0$ and $\sigma > 0$. Then the MLE for $\beta$ given by  
    \[
     \widehat{\beta}_T(X) = \frac{Z_T(X) - bT}{\int_0^T X_t dt}
     = \beta + \sigma \frac{\int_0^T \sqrt{X_t}dB_t}{\int_0^T X_tdt}
    \]
    is strongly consistent. Moreover, it is consistent in probability and asymptotically normal locally uniformly in the parameters. If the partition $\mathcal{P}_n$ satisfies $t_n^{(n)} \longrightarrow \infty$ and \eqref{eq: mesh size}, then the discretised estimator
    \[
        \widehat{\beta}^{(n)}(X) = \frac{Z_{t_n^{(n)}}^{\mathcal{P}_n} - t_n^{(n)}b}{\sum_{[u,v] \in \mathcal{P}_n} X_u (v-u) }
    \]
    is consistent in probability and asymptotically normal locally uniformly in the parameters. Remark that in this case, no assumptions of the form \eqref{eq: VCIR negative moment bound} are required.
\end{Remark}

In the next remark, we discuss the estimation of $b$ given that $\beta < 0$ is known. 

\begin{Remark}\label{remark: VCIR estimation b}
    Suppose that $\beta < 0$ is known and the same conditions as in Theorem \ref{thm: discrete MLE for VCIR} are satisfied. Then the MLE for $b$ is given by 
    \[
     \widehat{b}_T(X) = \frac{ \int_0^T \frac{1}{X_t}dZ_t(X) - \beta T}{\int_0^T \frac{1}{X_t}dt}
     = b + \sigma \frac{\int_0^T \frac{1}{\sqrt{X_t}} dB_t}{\int_0^T \frac{1}{X_t}dt}
    \]
    is strongly consistent, consistent in probability and asymptotically normal locally uniformly in the parameters. Moreover, its discretised version
    \[
        \widehat{b}^{(n,m)}(X) = \frac{ \sum_{[u,v] \in \mathcal{P}_n} X_u^{-1}\left( Z_v^{\mathcal{P}_m} - Z_u^{\mathcal{P}_m}\right)  - \beta t_n^{(n)}  }{ \sum_{[u,v] \in \mathcal{P}_n} X_u^{-1}(v-u) }
    \]
    is consistent in probability and asymptotically normal locally uniformly in the parameters. 
\end{Remark}

\subsection{Application to equidistant partitions}

For applications with discrete observations, it is natural to choose an equidistant partition where $t_k = k |\mathcal{P}_n|$. Below we briefly discuss conditions \eqref{eq: mesh size}, \eqref{eq: mesh size 1}, (K1), and (K2) in this context with particular focus on the choices $|\mathcal{P}_n| = n^{-\eta}$ with $\eta \in (0,1)$ whence $t_n^{(n)} = n^{1-\eta} \longrightarrow \infty$, and $|\mathcal{P}_n| = \frac{\log(n)}{n}$ whence $t_n^{(n)} = \log(n) \longrightarrow \infty$. These examples allow for a trade-off between long-horizon (big $t_n^{(n)}$) and discretisation depth.

\begin{Example}[fractional kernel]\label{exmpl:fractional_kernel}
    The fractional kernel $K(t) = t^{\alpha-1}/\Gamma(\alpha)$ with $\alpha \in (1/2,1)$ satisfies conditions (K1) and (K2) with $\gamma = \alpha - 1/2$. Moreover, we have $L(ds) = \frac{t^{-\alpha}}{\Gamma(1-\alpha)}ds$ and hence $L((0,t]) = \Gamma(\alpha)^{-1}t^{1-\alpha}$.
    \begin{enumerate}
    \item[(a)] Let $|\mathcal{P}_n| = n^{-\eta}$ with $\frac{1}{2\alpha} < \eta < 1$. Since $2\alpha \eta > 1$, condition \eqref{eq: mesh size} is satisfied while condition \eqref{eq: mesh size 1} holds for any sequence $(m(n))_{n \geq 1}$ with the property 
    \[
        \lim_{n \to \infty} \frac{n^{1 + (1-\eta)\left(\frac{3}{2}-\alpha\right)}}{m(n)^{\left(\alpha - \frac{1}{2}\right)\eta}} = 0.
    \]

    \item[(b)] Let $|\mathcal{P}_n| = \frac{\log(n)}{n}$. Then conditions \eqref{eq: mesh size} holds, and \eqref{eq: mesh size 1} reduces to 
    \[
        \lim_{n \to \infty}\left(\frac{\log(m(n))}{\log(n)}\right)^{\alpha - 1/2}\frac{n}{m(n)^{\alpha - 1/2}} = 0.
    \]
    \end{enumerate}
\end{Example}

The next example illustrates the behaviour of a kernel that is not regular but essentially behaves like the fractional kernel with $\alpha \nearrow 1$.

\begin{Example}[log-kernel]\label{exmpl:log_kernel}
    Let $K(t) = \log(1+1/t)$. Then conditions (K1) and (K2) are satisfied with any choice of $\gamma \in (0,1/2)$. Then \cite[Remark 5.1]{BFK24} gives $L((0,t]) \leq C \log(1+1/t)^{-1}$.
    \begin{enumerate}
        \item[(a)] Let $|\mathcal{P}_n| = n^{-\eta}$ for $\frac{1}{2} < \eta < 1$, then condition \eqref{eq: mesh size} is satisfied and \eqref{eq: mesh size 1} is satisfied if for some $\gamma \in (0,1/2)$
    \begin{align*}
        \lim_{n \to \infty}n^{\frac{3 - \eta}{2}} m(n)^{- \eta \gamma} = 0. 
    \end{align*}

        \item[(b)] Let $|\mathcal{P}_n| = \frac{\log(n)}{n}$. Then \eqref{eq: mesh size} holds and \eqref{eq: mesh size 1} is satisfied if for some $\gamma \in (0,1/2)$
        \[
        \lim_{n \to \infty} n \left( \frac{\log(m(n))}{m(n)}\right)^{\gamma}. = 0.
        \]
    \end{enumerate}
\end{Example}

Below, we illustrate our conditions for the regular kernel obtained as a linear combination of exponentials. The latter plays a central role in multi-factor Markovian approximations, see \cite{MR3934104, MR4521278}.

\begin{Example}[exponential kernel]\label{exmpl:exp_kernel}
    Let $K(t) = \sum_{i=1}^{N}c_i e^{-\lambda_i t}$ for some $c_1,\dots, c_N > 0$ and $\lambda_1,\dots, \lambda_N \geq 0$. Then condition (K1) holds for $\gamma = 1/2$ and (K2) satisfied. Using \cite[Remark 5.1]{BFK24} we obtain:
    \begin{enumerate}
        \item[(a)] Let $|\mathcal{P}_n| = n^{-\eta}$ with $0 < \eta < 1$, then \eqref{eq: mesh size} holds, and \eqref{eq: mesh size 1} is implied by
        \[
            \lim_{n \to \infty} n^{\frac{1+\eta}{2}}m(n)^{-\frac{1}{2}} = 0.
        \]

        \item[(b)] Let $|\mathcal{P}_n| = \frac{\log(n)}{n}$. Then \eqref{eq: mesh size} holds and \eqref{eq: mesh size 1} is satisfied if 
        \[
        \lim_{n \to \infty} n \log(n)^{\frac{1}{2}} \left( \frac{\log(m(n))}{m(n)}\right)^{\frac{1}{2}} = 0.
        \]
    \end{enumerate}
\end{Example}

Similarly, we could also consider kernels, e.g., of the form $\log(1 +  t^{-\alpha})$ and $e^{-t^{\alpha}}$ with $\alpha \in (0,1)$, and also $\log(1+t)/t$.

\section{Numerical experiments}

In this section, we perform numerical experiments that confirm our theoretical findings. To be more precise, we demonstrate the accuracy of the discrete high-frequency estimation for the shifted fractional kernel $K_{\alpha}(t) = (t + \varepsilon)^{\alpha - 1}/\Gamma(\alpha)$, where $\varepsilon \geq 0$ is a regularising shift. As a first step, we simulate sample paths of the Volterra Cox–Ingersoll–Ross process \eqref{eq: VCIR} using a discrete-time Euler-type scheme adapted to the memory structure of Volterra equations.
We employ an equidistant time discretisation $t_k^{(n)} = \frac{k}{n}T$ for a fixed time horizon $T > 0$ and a fixed number of time steps $n \in \mathbb{N}$. We write $t_k := t_k^{(n)}$ for ease of notation. A known drawback of applying Euler-type schemes to square-root processes is that intermediate approximations can become negative. To mitigate this issue and preserve the non-negativity of the process, we use the following truncated Euler scheme:
\begin{align*}
    \widehat{X}_{k+1} = \left( x_0 + \frac{T}{n} \sum_{i=0}^{k} K_{\alpha}(t_{k+1} - t_i) \big( b + \beta \widehat{X}_i \big) + \sigma \sqrt{\frac{T}{n}} \sum_{i=0}^{k} K_{\alpha}(t_{k+1} - t_i) \sqrt{(\widehat{X}_i)_+} \, \xi_{i+1} \right)_+,
\end{align*}
where $\xi_1, \dots, \xi_n$ is a sequence of independent standard Gaussian random variables, and $(x)_+ = \max\{0, x\}$ denotes the positive part. Using $(\cdot)_+$ ensures the simulated process remains non-negative. For further details on numerical simulation methods in this context, we refer to \cite{MR3934104, alfonsi2023, alfonsi2022, MR4521278, BBD08}.

Recall that the numerical approximation of $Z(X)$ given by \eqref{eq: Z process} appears for the parameter estimation in the MLE. While our estimates from Section 3 justify its convergence, it turns out that it is highly sensitive to $\alpha$. For smaller values of $\alpha$, convergence becomes much slower, and the discretisation depth needs to be strongly increased to obtain desirable results. To circumvent this issue and obtain a clear separation between the error arising from the estimation procedure and the error introduced by approximating $Z(X)$ via \eqref{eq: Z process}, below we introduce directly the discredited semimartingale process $\widehat{Z}_{k}$ via
\begin{align*}
    \widehat{Z}_{k+1} = \widehat{Z}_0 + \sum_{i=0}^{k} \left( b + \beta \widehat{X}_i \right) \frac{T}{n} + \sigma \sum_{i=0}^{k} \sqrt{(\widehat{X}_i)_+} \sqrt{\frac{T}{n}} \, \xi_{i+1}.
\end{align*}
This formulation aggregates the drift and diffusion contributions using the same Brownian increments as in the simulation of $\widehat{X}_k$. It would be interesting, but certainly outside the scope of this work, to study different numerical schemes that would provide faster and more robust convergence results.

For fixed $x_0, b, \beta, \sigma$ given by
\[
    x_0 = 1, \ \ b = 1.2, \ \ \beta = -1, \ \ \sigma = 1.2,
\]
we generate for different choices of $T$, $n$, and Volterra kernels, $N = 1.000$ sample paths. 
For each of these sample paths, we evaluate the discretised maximum-likelihood estimators $(\widehat{b}^{\textbf{MLE}}, \widehat{\beta}^{\textbf{MLE}})$ as described in Theorem \ref{thm: discrete MLE for VCIR}, and the method of moments estimators $(\widehat{b}^{\textbf{MoM}}, \widehat{\beta}^{\textbf{MoM}})$ for the true fractional kernel described in Corollary \ref{cor: VCIR method of moments}. Our results are summarised in the tables \ref{table:VCIR_T_shift}--\ref{table:VCIR_alpha}, of which the following table contains a rough summary.

\begin{table}[ht]
\centering
\begin{tabular}{|l|p{5.5cm}|p{5.5cm}|}
\hline
\textbf{Metric} & \textbf{MLE} (different kernels) & \textbf{MoM} (true fractional kernel) \\
\hline
Accuracy for $\hat{b}$ & Improves rapidly with $T$ \newline very accurate for large $T$ & Biased for small $T$ \newline accurate for large $T$ \\
\hline
Accuracy for $\hat{\beta}$ & Good with large $T$ \newline slightly worse than $\hat{b}$ & Larger errors at small $T$ \newline improves for large $T$ \\
\hline
Effect of $T$ & Major improvement for MLE & Strong improvement with $T$ \\
\hline
Effect of $\Delta t_{k+1}$ & Minor effect & Slight improvement \\
\hline
Effect of $\alpha$ & Very robust & Sensitive: requires much larger $T$ for small $\alpha$ (long memory) \\
\hline
Effect of shift & Minimal effect; still robust & \\
\hline
\end{tabular}
\caption{Comparison of MLE and MoM Estimators}
\label{tab:mle_mom_comparison}
\end{table}

From these tables, we observe that the MLE consistently outperforms the MoM across a wide range of parameter settings. The MLE exhibits strong robustness with respect to variations in the discretisation depth $\Delta t_{k+1}$, the time horizon $T$, the Hurst index $\alpha$, and even under the presence of additive shifts in the fractional kernel. In particular, the relative estimation errors of $\widehat{b}^{\textbf{MLE}}$ and $\widehat{\beta}^{\textbf{MLE}}$ decrease rapidly with increasing $T$ as demonstrated in Figures \ref{fig:VCIR_b_estimate}--\ref{fig:VCIR_beta_estimate}. We observe that convergence holds regardless of whether the boundary condition \eqref{eq: negative moments} is violated or not. Hence, we suspect that more sophisticated sequential estimation methods may guarantee some weaker forms of convergence beyond condition \eqref{eq: negative moments}. 

Concerning the MoM estimator for the true fractional kernel \eqref{eq: fractional kernel} (i.e. shift $\varepsilon=0$), as given in Corollary \ref{cor: VCIR method of moments}, the numerical performance improves markedly with increasing time horizon $T$, while it remains biased for small $T$ (cf. Table \ref{table:VCIR_T_dt_MOM}). A moderate improvement is also observed with finer discretisation $\Delta t_{k+1}$ (cf. Figures \ref{fig:VCIR_b_MOM_estimate} and \ref{fig:VCIR_beta_MOM_estimate})). These observations align with our theoretical consistency results, which apply to both high- and low-frequency observations. Moreover, the parameter $\alpha$ plays a crucial role in the convergence rate: for values close to 1, accurate estimates are achieved even at moderate $T$, whereas for smaller $\alpha$, substantially larger time horizons are needed to capture the long-memory structure and obtain stable estimates. This indicates sensitivity to $\alpha$ caused by the slower rate of convergence for the Law of Large Numbers.

It remains an open question to precisely quantify the impact of discretisation schemes on parameter convergence, particularly for non-semimartingale implementations for the discretised process $Z^{\mathcal{P}_m}$. In \cite{BK13}, the authors were able to remove these integrals by a suitable application of the Ito formula. A similar method seems not to apply in our case. Further research beyond the scope of this work may shed more light on the convergence of the numerical schemes on large time-scales.

\appendix 
\section{Uniform weak convergence} \label{sec:uniform weak convergence}

Let $(E, d)$ be a complete separable metric space and $\mathcal{B}(E)$ the Borel-$\sigma$ algebra on $E$. Let $p \geq 0$ and let $\mathcal{P}_p(E)$ be the space of Borel probability measures on $E$ such that $\int_E d(x,x_0)^p \mu(dx) < \infty$ holds for some fixed $x_0 \in E$. Here and below, we let $\Theta$ be an abstract index set (of parameters). 

\begin{Definition}
    Let $p \geq 0$ and $(\mu^{\theta}_n)_{n \geq 1, \theta \in \Theta}, (\mu^{\theta})_{\theta \in \Theta} \subset \mathcal{P}_p(E)$. Then $(\mu^{\theta}_n)_{n \geq 1, \theta \in \Theta}$ converges weakly in $\mathcal{P}_p(E)$ to $(\mu^{\theta})_{\theta \in \Theta}$ uniformly on $\Theta$, if
    \begin{align}\label{eq: weak convergence}
        \lim_{n \to \infty} \sup_{\theta \in \Theta}\left|\int_E f(x)\mu_n^{\theta}(dx) - \int_{E}f(x)\mu^{\theta}(dx) \right| = 0
    \end{align}
    holds for each continuous function $f: E \longrightarrow \R$ such that there exists $C_f > 0$ with $|f(x)| \leq C_f(1 + d(x_0,x)^p)$ for $x \in E$.
\end{Definition}

We abbreviate the weak convergence on $\mathcal{P}_p(E)$ uniform on $\Theta$ by $\mu_n^{\Theta} \Lu_p \mu^{\theta}$. In the particular case $p=0$ we write $\mathcal{P}_0(E) = \mathcal{P}(E)$ and denote the convergence by $\mu_n^{\theta} \Lu \mu^{\theta}$ which corresponds to \eqref{eq: weak convergence} for bounded and continuous functions. The next proposition characterises the weak convergence in $\mathcal{P}(E)$ under an additional tightness condition.

\begin{Proposition}\label{prop: 1}
    Let $(\mu^{\theta}_n)_{n \geq 1, \theta \in \Theta}, (\mu^{\theta})_{\theta \in \Theta} \subset \mathcal{P}(E)$ and suppose that both families of probability measures are tight, i.e. for each $\e > 0$ there exists a compact $K \subset E$ such that
    \begin{align}\label{eq: tightness 2}
     \sup_{\theta \in \Theta}\sup_{n \geq 1}\mu_n^{\theta}(K^{c}) + \sup_{\theta \in \Theta}\mu^{\theta}(K^c) < \e.
    \end{align}
    Then the following assertions are equivalent:
    \begin{enumerate}
        \item[(a)] $\mu_n^{\theta} \Lu \mu^{\theta}$.
        \item[(b)] \eqref{eq: weak convergence} holds for each bounded and uniformly continuous function $f: E \longrightarrow \R$.
        \item[(c)] \eqref{eq: weak convergence} holds for each bounded and Lipschitz continuous function $f: E \longrightarrow \R$.
    \end{enumerate}
\end{Proposition}
\begin{proof}
    The implications $(a) \Rightarrow (b) \Rightarrow (c)$ are clear. Let us prove $(c) \Rightarrow (a)$. Let $f: E \longrightarrow \R$ be continuous and bounded, and let $\e > 0$ be arbitrary. Choose $K_{\e} \subset E$ compact as in the tightness condition. Then there exists\footnote{apply e.g. the Stone-Weierstrass theorem for the subalgebra $A = \{1\} \cup \{d_x : x \in E\}$ where $d_x(y) = \min\{1, d(x,y)\}$ of $C_b(E)$ equipped with the compact-open topology. Hence the space of bounded Lipschitz continuous functions is dense in $C_b(E)$ with respect to the compact open topology.} a Lipschitz continuous function $g_{K_{\e}}: E \longrightarrow \R$ such that $\sup_{x \in K_{\e}}|f(x) - g_{K_{\e}}(x)| < \e$ and $\| g_{K_{\e}}\|_{\infty} \leq \|f\|_{\infty} + 1$. Then we obtain from the tightness condition 
    \begin{align*}
        \left| \int_{E}fd\mu_n^{\theta} - \int_E f d\mu^{\theta}\right|
        &\leq \left| \int_E g_{K_{\e}} d\mu_n^{\theta} - \int_E g_{K_{\e}} d\mu^{\theta}\right|
        \\ &\qquad + \left| \int_{K_{\e}} (f-g_{K_{\e}}) d\mu_n^{\theta}\right| + \left|\int_{K_{\e}} (f-g_{K_{\e}}) d\mu^{\theta}\right|
        \\ &\qquad + \left| \int_{K_{\e}^c} (f-g_{K_{\e}}) d\mu_n^{\theta}\right| + \left|\int_{K_{\e}^c} (f-g_{K_{\e}}) d\mu^{\theta}\right|
        \\ &\leq \left| \int_E g_{K_{\e}} d\mu_n^{\theta} - \int_E g_{K_{\e}} d\mu^{\theta}\right| + 2\e + (2\|f\| + 1)\e 
    \end{align*}
    Since $g_{K_{\e}}$ is Lipschitz continuous, we obtain from (c) 
    \[
     \limsup_{n \to \infty}\sup_{\theta \in \Theta}\left| \int_{E}fd\mu_n^{\theta} - \int_E f d\mu^{\theta}\right|
     \leq 2\e + (2\|f\| + 1)\e.
    \]
    Letting $\e \searrow 0$ proves $\mu_n^{\theta} \Lu \mu^{\theta}$.
\end{proof}

In the next proposition, we provide a characterisation for convergence in $\mathcal{P}_p(E)$.

\begin{Proposition}\label{prop: P characterisation}
    Let $(\mu^{\theta}_n)_{n \geq 1}, (\mu^{\theta})_{\theta \in \Theta} \subset \mathcal{P}_p(E)$ with $p \in (0,\infty)$, and suppose that
    \begin{align}\label{eq: tightness 1}
        \lim_{R \to \infty}\sup_{\theta \in \Theta}\int_E \1_{\{d(x,x_0) > R\}}d(x,x_0)^p \mu^{\theta}(dx) = 0
    \end{align}
    Then the following are equivalent:
    \begin{enumerate}
        \item[(a)] $\mu_n^{\theta} \Lu_p \mu^{\theta}$. 

        \item[(b)] $(\mu^{\theta}_n)_{n \geq 1} \Lu \mu^{\theta}$ and
        \[
         \lim_{n \to \infty}\sup_{\theta \in \Theta}\left| \int_E d(x,x_0)^p \mu_n^{\theta}(dx) - \int_E d(x,x_0)^p \mu^{\theta}(dx)\right| = 0.
        \]

        \item[(c)] $(\mu^{\theta}_n)_{n \geq 1} \Lu \mu^{\theta}$ and
        \[
         \lim_{R \to \infty}\limsup_{n \to \infty}\sup_{\theta \in \Theta}\int_E \1_{\{d(x,x_0) > R\}}d(x,x_0)^p \mu_n^{\theta}(dx) = 0.
        \]
    \end{enumerate}
\end{Proposition}
\begin{proof}
    $(a) \Rightarrow (b)$. This is clear. 
    
    $(b) \Rightarrow (c)$. Take $R > 0$ and let $\varphi_R: E \longrightarrow [0,1]$ be continuous with $\varphi_R(x) = 1$ for $d(x_0,x) \leq R$ and $\varphi_R(x) = 0$ for $d(x_0,x) > R+1$. Then $\1_{\{d(x_0,x) > R+1\}} = \1_{\{d(x_0,x) > R+1\}}(1 - \varphi_R(x))$ and $1 - \varphi_R(x) \leq \1_{\{d(x_0,x) > R\}}$, and hence
    \begin{align*}
        &\ \int_E d(x_0,x)^p \1_{\{d(x_0,x) > R+1\}}\mu_n^{\theta}(dx)
        \\ &\leq \int_E d(x_0,x)^p (1 - \varphi_R(x))\mu_n^{\theta}(dx)
        \\ &= \int_E d(x_0,x)^p \left( \mu_n^{\theta}(dx) - \mu^{\theta}(dx)\right)
         + \int_E d(x_0,x)^p \left( 1 - \varphi_R(x) \right) \mu^{\theta}(dx)
        \\ &\qquad + \int_E d(x_0,x)^p \varphi_R(x) \left( \mu^{\theta}(dx) - \mu_n^{\theta}(dx)\right) 
        \\ &\leq  \left|\int_E d(x_0,x)^p \left(\mu_n^{\theta}(dx) - \mu^{\theta}(dx)\right) \right|
        \\ &\qquad +  \int_E d(x_0,x)^p \1_{\{d(x_0,x) > R\}} \mu^{\theta}(dx)
        \\ &\qquad + \left| \int_E d(x_0,x)^p \varphi_R(x) \left(\mu^{\theta}(dx) - \mu_n^{\theta}(dx) \right) \right|.
    \end{align*}
    Taking the supremum over $\theta \in \Theta$ and then $\limsup_{n \to \infty}$ shows that the first and the last term converge to zero due to assumption (b). Finally, letting $R \to \infty$, one finds that the second term tends to zero by tightness of $(\mu^{\theta})_{\theta \in \Theta}$. 
    
    $(c) \Rightarrow (a)$. Let $f: E \longrightarrow \R$ be continuous such that $|f(x)| \leq C_f(1 + d(x_0,x)^p)$ holds for $x \in E$. Let $R > 0$ and $\varphi_R$ be given as above. Using $1 - \varphi_R(x) \leq \1_{\{d(x_0,x) > R\}}$ we obtain
    \begin{align*}
        \left| \int_E f d\mu_n^{\theta} - \int_E f d\mu^{\theta}\right|
        &\leq \left| \int_E \varphi_R f d\mu_n^{\theta} - \int_E \varphi_R f d\mu^{\theta}\right| 
        \\ &\qquad + C_f \int_E \1_{\{d(x_0,x) > R\}}\left( 1 + d(x_0,x)^p\right)\left( \mu_n^{\theta}(dx) + \mu^{\theta}(dx)\right)
    \end{align*}
    Taking the supremum over $\theta \in \Theta$, then $\limsup_{n \to \infty}$, and finally $R \to \infty$, proves property (a).
\end{proof}

For applications to limit theorems, we need analogues of the continuous mapping theorem and Slutsky's theorem where convergence is uniform on $\Theta$. To simplify the notation, we formulate it with respect to random variables $X_n^{\theta}, X^{\theta}$ defined on some probability space $(\Omega, \F, \P)$. Thus, we say that $X_n^{\theta} \Lu X^{\theta}$ if $\mathcal{L}(X_n^{\theta}) \Lu \mathcal{L}(X^{\theta})$. Finally, we say that $X_n^{\theta} \longrightarrow X^{\theta}$ in probability uniformly on $\Theta$, if 
\[
 \lim_{n \to \infty}\sup_{\theta \in \Theta}\P[ d(X_n^{\theta}, X^{\theta}) > \e] = 0, \qquad \forall \e > 0.
\]

As a first step, we need the following lemma. 

\begin{Lemma}\label{lemma: product}
    Let $(E_0,d_0), (E_1,d_1)$ be complete separable metric spaces. Suppose that  $(\mu_n^{\theta})_{n \geq 1, \theta \in \Theta}, (\mu^{\theta})_{\theta \in \Theta} \subset \mathcal{P}(E_0 \times E_1)$ satisfy \eqref{eq: tightness 2} on $E = E_0 \times E_1$. If 
    \[
     \lim_{n \to \infty}\sup_{\theta \in \Theta}\left| \int_{E_0 \times E_1} f(x)g(y) \mu_n^{\theta}(dx,dy) - \int_{E_0 \times E_1} f(x)g(y)\mu^{\theta}(dx,dy) \right| = 0
    \]
    holds for all bounded and Lipschitz continuous functions $f: E_0 \longrightarrow \R$ and $g: E_1 \longrightarrow \R$, then $\mu_n^{\theta} \Lu \mu^{\theta}$.
\end{Lemma}
\begin{proof}
    An application of the Stone-Weierstrass theorem to  
    \[
     A = \{ h = f \otimes g \ : \ f \in \mathrm{Lip_b}(E_0), \ g \in \mathrm{Lip_b}(E_1) \}
    \]
    shows that $\mathrm{lin}(A)$ is dense in $C_b(E_0 \times E_1)$ with respect to the compact open topology. The proof can be deduced in the same way as Proposition \ref{prop: 1}.
\end{proof}

Next, we provide some simple versions of the continuous mapping theorem adapted towards convergence uniform on $\Theta$.

\begin{Theorem}[Uniform continuous mapping theorem]\label{thm: uniform continuous mapping theorem}\label{thm: uniform CMT}
 Let $(X_n^{\theta})_{n \geq 1, \theta \in \Theta}$ and $(X^{\theta})_{\theta \in \Theta}$ be $E$-valued random variables on some probability space $(\Omega, \F, \P)$. Let $F: E \longrightarrow E'$ be measurable where $(E',d')$ is another complete and separable metric space. The following assertions hold:
 \begin{enumerate}
     \item[(a)] If $X_n^{\theta} \longrightarrow c(\theta)$ in probability uniformly on $\Theta$, $c(\theta)$ is deterministic and $F$ is uniformly continuous in $(c(\theta))_{\theta \in \Theta}$, then $F(X_n^{\theta}) \longrightarrow F(c(\theta))$ in probability uniformly on $\Theta$.

    \item[(b)] If $X_n^{\theta} \longrightarrow X^{\theta}$ in probability uniformly on $\Theta$ and $F$ is uniformly continuous, then $F(X_n^{\theta}) \longrightarrow F(X^{\theta})$ in probability uniformly on $\Theta$. 

    \item[(c)] Suppose there exists $\mu \in \mathcal{P}(E)$ such that for each $\e > 0$ there exists $\delta > 0$ with
    \begin{align}\label{eq: locally absolute continuity}
     \mu(A) < \delta \quad \Longrightarrow \quad \sup_{\theta \in \Theta}\P[X^{\theta} \in A] < \e.
    \end{align}
    for any Borel set $A \subset E$. If $X_n^{\theta} \Longrightarrow X^{\theta}$ and the set of discontinuity points $D_F \subset E$ of $F$ satisfies $\P[X^{\theta} \in D_F] = 0$, then $F(X_n^{\theta}) \Lu F(X^{\theta})$.
 \end{enumerate}
\end{Theorem}
\begin{proof}
 (a) Let $\e > 0$. By uniform continuity of $F$ we find $\delta > 0$ independent of $\theta$ such that $d'(F(x), c(\theta)) \leq \e$ holds for $d(x,c(\theta)) < \delta$. The assertion follows from
 \begin{align*}
     \P\left[ d'(F(X_n^{\theta}), c(\theta)) > \e \right]
     &= \P\left[ d'(F(X_n^{\theta}), c(\theta)) > \e, \ d(X_n^{\theta}, c(\theta)) \geq \delta \right]
     \\ &\leq \P\left[ d(X_n^{\theta}, c(\theta)) \geq \delta \right]
 \end{align*}

 (b) The proof is essentially the same as in part (a) and is therefore omitted.

 (c) Let $B \subset E'$ be closed. Then 
  \begin{align*}
     &\ \limsup_{n\to \infty}\sup_{\theta \in \Theta}\left(\P[ F(X_n^{\theta}) \in B] - \P[F(X^{\theta}) \in B] \right)
     \\ &\leq \limsup_{n\to \infty}\sup_{\theta \in \Theta}\left(\P[ X_n^{\theta} \in \overline{F^{-1}(B)}] - \P[X^{\theta} \in \overline{F^{-1}(B)}] \right) 
     \\ &\qquad + \sup_{\theta \in \Theta}\left(\P[ X^{\theta} \in \overline{F^{-1}(B)}] - \P[X^{\theta} \in F^{-1}(B)] \right) 
 \end{align*}
 Since $X_n^{\theta} \Lu X^{\theta}$, the first term is bounded by zero by the uniform Portmanteau theorem (see \cite[Theorem 2.1]{BH19}). For the second term we use $\overline{F^{-1}(B)} \subset F^{-1}(B) \cup D_F$ to find 
 \begin{align*}
     \sup_{\theta \in \Theta}\left(\P[ X^{\theta} \in \overline{F^{-1}(B)}] - \P[X^{\theta} \in F^{-1}(B)] \right)
     &\leq \sup_{\theta \in \Theta}\P[X^{\theta} \in D_F] = 0.
 \end{align*}
 The assertion now follows from \cite[Theorem 2.1]{BH19}.
\end{proof}

\begin{Proposition}\label{prop: uniform weak convergence}
 Let $(X_n^{\theta})_{n \geq 1, \theta \in \Theta}$, $(Y_n^{\theta})_{n \geq 1, \theta \in \Theta}$, $(X^{\theta})_{\theta \in \Theta}$, and $(Y^{\theta})_{n \geq 1, \theta \in \Theta}$ be $E$-valued random variables on some probability space $(\Omega, \F, \P)$. The following assertions hold:
 \begin{enumerate}
     \item[(a)] If $X_n^{\theta} \Lu X^{\theta}$, $(Y_n^{\theta})_{n \geq 1, \theta \in \Theta}$ satisfies
     \[
      \lim_{n \to \infty}\sup_{\theta \in \Theta}\P[ d(X_n^{\theta}, Y_n^{\theta}) > \e] = 0, \qquad \forall \e > 0,
     \]
     and $\mu_n^{\theta} = \mathcal{L}(X_n^{\theta})$, $\mu^{\theta} = \mathcal{L}(X^{\theta})$ satisfy \eqref{eq: tightness 2}, then $Y_n^{\theta} \Lu X^{\theta}$.

     \item[(b)] Suppose that $X_n^{\theta} \Lu X^{\theta}$, $Y_n^{\theta} \longrightarrow c(\theta)$ in probability uniformly on $\Theta$ with deterministic $(c(\theta))_{\theta \in \Theta}$, and for each $\e > 0$ there exists a compact $K \subset E$ such that $c(\theta) \not \in K^c$ for each $\theta \in \Theta$ and
     \begin{align}\label{eq: tightness 3}
      \sup_{n \geq 1}\sup_{\theta \in \Theta}\left(\P[X_n^{\theta} \in K^c] + \P[X^{\theta} \in K^c] + \P[Y_n^{\theta} \in K^c] \right) < \e.
     \end{align}
     Then $(X_n^{\theta}, Y_n^{\theta}) \Lu (X^{\theta}, c(\theta))$.
 \end{enumerate}
\end{Proposition}
\begin{proof}
     (a) Let $f: E \longrightarrow \R$ be bounded and Lipschitz continuous with constant $C_f$. Then
 \begin{align*}
     \left| \E[ f(Y_n^{\theta})] - \E[ f(X^{\theta})] \right|
     &\leq \left| \E[ f(Y_n^{\theta})] - \E[ f(X_n^{\theta})] \right| + \left| \E[ f(X_n^{\theta})] - \E[ f(X^{\theta})] \right|
     \\ &\leq C_f \eta + C_f \P\left[ d(Y_n^{\theta}, X_n^{\theta}) > \eta \right] + \left| \E[ f(X_n^{\theta})] - \E[ f(X^{\theta})] \right|
 \end{align*}
 for each $\eta > 0$. Using Proposition \ref{prop: 1}, we readily deduce the assertion. 

 (b) Let $f,g: E \longrightarrow \R$ be bounded and Lipschitz continuous with constants $C_f, C_g$. Then we obtain for $\e > 0$
    \begin{align*}
       &\ \left| \E[ f(X_n^{\theta})g(Y_n^{\theta}) ] - \E[ f(X^{\theta})g(c(\theta)) ]\right|
        \\ &\leq \left| \E[ f(X_n^{\theta}) g(Y_n^{\theta}) ] - \E[ f(X_n^{\theta})g(c(\theta)) ]\right|
         + \left| \E[ f(X_n^{\theta})g(c(\theta)) ] - \E[ f(X^{\theta})g(c(\theta)) ]\right|
         \\ &\leq \|f\|_{\infty}C_g \E \left[ 1 \wedge d(Y_n^{\theta}, c(\theta)) \right]
          + C_f \| g\|_{\infty} \left| \E[ f(X_n^{\theta}) ] - \E[ f(X^{\theta})]\right|
          \\ &\leq \|f\|_{\infty}C_g \e + C_g \P\left[ d(Y_n^{\theta}, c(\theta)) > \e \right]
          + C_f \| g\|_{\infty} \left| \E[ f(X_n^{\theta}) ] - \E[ f(X^{\theta})]\right|
    \end{align*}
    Taking the supremum over $\Theta$ and letting then $n \to \infty$ gives
    \[
     \limsup_{n\to \infty} \sup_{\theta \in \Theta}\left| \E[ f(X_n^{\theta})g(Y_n^{\theta}) ] - \E[ f(X^{\theta})g(c(\theta)) ]\right| \leq \|f\|_{\infty}C_g \e.
    \]
    Letting $\e \searrow 0$ proves, given Lemma \ref{lemma: product}, the assertion. Hereby assumption \eqref{eq: tightness 2} is satisfied due to $\P[ (X_n^{\theta}, Y_n^{\theta}) \in (K \times K)^c ] \leq \P[X_n^{\theta} \in K^c] + \P[Y_n^{\theta} \in K^c]$ and similarly $\P[(X^{\theta}, c(\theta)) \in (K \times K)^c] \leq \P[ X^{\theta} \in K^c] + \P[c(\theta) \in K^c]$ combined with assumption \eqref{eq: tightness 3}.
\end{proof}

Finally, let us prove a uniform weak convergence result for probability measures supported on the positive cone $\R_+^d = \{ x \in \R^d \ : \ x_k \geq 0, k = 1,\dots, d\}$ in terms of Laplace transforms.

\begin{Theorem}\label{thm: uniform Laplace transform convergence}
    Let $(\mu_{n}^{\theta})_{n \geq 1, \theta \in \Theta}$ and $(\mu^{\theta})_{\theta \in \Theta}$ be Borel probability measures on $\R_+^d$. Suppose that
    \begin{align*}
        \lim_{n \to \infty}\sup_{\theta \in \Theta}\left| \int_{\R_+^d} e^{-\langle z,x\rangle}\mu_n^{\theta}(dx) - \int_{\R_+^d}e^{-\langle z,x\rangle}\mu^{\theta}(dx)\right| = 0
    \end{align*}
    and for each $\e>0$ there exists $R > 0$ such that $\sup_{\theta \in \Theta}\mu^{\theta}(\{|x| > R\}) < \e$. Then $\mu_n^{\theta} \Lu \mu^{\theta}$.
\end{Theorem}
\begin{proof}
    Let $f: \R^d \longrightarrow \R$ be continuous and bounded. Fix $\e > 0$ and let $R > 0$ such that $\mu^{\theta}(\{|x| > R\}) < \e$ for each $\theta \in \Theta$. Again by the Stone-Weierstrass theorem there exists $g$ being the finite linear combination of $e^{-\langle z,\cdot\rangle}$ such that $\sup_{|x| \leq R}|f(x) - g(x)| < \e$. Moreover, we may take $g$ is such a way that $\|g\|_{\infty} \leq \|f\|_{\infty} + 1$. The assertion can be deduced similarly to the proof of Proposition \ref{prop: 1}.
\end{proof}

\section{Supplementary material: Tables and figures}

\begin{table}[H]
\centering
\begin{minipage}[b]{0.4\textwidth}
\begin{tabular}{ |p{1.0cm}|p{1.0cm}||p{1.6cm}|p{1.6cm}|p{1.6cm}|p{1.6cm}|}
    \hline
    \multicolumn{4}{|c|}{$\Delta t_{k+1}=0.05$ and $\alpha = 0.75$} \\
    \hline
    $T$ &
    $\varepsilon$ &
    rel. error $\widehat{b}^{\text{MLE}}$ & 
    rel. error $\widehat{\beta}^{\text{MLE}}$ \\
    \hline
    10 & 0.5 & 0.3202 & 0.3518    \\
    20 & 0.5 & 0.1579 & 0.2017    \\
    30 & 0.5 & 0.0829 & 0.1284    \\
    40 & 0.5 & 0.0538 & 0.0878    \\
    50 & 0.5 & 0.0385 & 0.0750    \\
    60 & 0.5 & 0.0268 & 0.0567    \\
    70 & 0.5 & 0.0195 & 0.0494    \\
    80 & 0.5 & 0.0146 & 0.0408    \\
    90 & 0.5 & 0.0132 & 0.0374    \\
    100 & 0.5 & 0.0081 & 0.0295   \\
    \hdashline
    10 & 0.6 & 0.3447 & 0.3724    \\
    20 & 0.6 & 0.1728 & 0.2151    \\
    30 & 0.6 & 0.0918 & 0.1361    \\
    40 & 0.6 & 0.0635 & 0.0958    \\
    50 & 0.6 & 0.0479 & 0.0836    \\
    60 & 0.6 & 0.0326 & 0.0618    \\
    70 & 0.6 & 0.0268 & 0.0559    \\
    80 & 0.6 & 0.0209 & 0.0462    \\
    90 & 0.6 & 0.0180 & 0.0416    \\
    100 & 0.6 & 0.0128 & 0.0338   \\
    \hdashline
    10 & 0.7 & 0.3657 & 0.3896   \\
    20 & 0.7 & 0.1852 & 0.2259    \\
    30 & 0.7 & 0.0998 & 0.1431    \\
    40 & 0.7 & 0.0709 & 0.1022    \\
    50 & 0.7 & 0.0533 & 0.0885    \\
    60 & 0.7 & 0.0398 & 0.0687    \\
    70 & 0.7 & 0.0319 & 0.0604    \\
    80 & 0.7 & 0.0270 & 0.0520    \\
    90 & 0.7 & 0.0221 & 0.0454    \\
    100 & 0.7 & 0.0173 & 0.0381   \\
    \hline
\end{tabular}
\captionsetup{width=1\linewidth}
\caption{VCIR process with shifted fractional kernel: varying $T$ and $\varepsilon$.}
\label{table:VCIR_T_shift}
\end{minipage}
  \hfill
\begin{minipage}[b]{0.4\textwidth}
\begin{tabular}{ |p{1.0cm}|p{1.0cm}||p{1.6cm}|p{1.6cm}|p{1.6cm}|p{1.6cm}|}
    \hline
    \multicolumn{4}{|c|}{$\varepsilon=0$ and $\alpha = 0.75$} \\
    \hline
    $T$ &
    $\Delta t_{k+1}$ &
    rel. error $\widehat{b}^{\text{MoM}}$ & 
    rel. error $\widehat{\beta}^{\text{MoM}}$ \\
    \hline
    50 & 0.50 & 0.2829 & 0.3599 \\
    100 & 0.50 & 0.1638 & 0.1953 \\
    150 & 0.50 & 0.1295 & 0.1523 \\
    200 & 0.50 & 0.1163 & 0.1363 \\
    250 & 0.50 & 0.1074 & 0.1258 \\
    300 & 0.50 & 0.0983 & 0.1140 \\
    \hdashline
    50 & 0.10 & 0.2097 & 0.2950 \\
    100 & 0.10 & 0.1227 & 0.1705 \\
    150 & 0.10 & 0.0917 & 0.1253 \\
    200 & 0.10 & 0.0664 & 0.0907 \\
    250 & 0.10 & 0.0533 & 0.0723 \\
    300 & 0.10 & 0.0494 & 0.0665 \\
    \hdashline
    50 & 0.05 & 0.1630 & 0.2544 \\
    100 & 0.05 & 0.0664 & 0.1100 \\
    150	& 0.05 & 0.0372 & 0.0677 \\
    200 & 0.05 & 0.0213 & 0.0434 \\
    250 & 0.05 & 0.0103 & 0.0280 \\
    300 & 0.05 & 0.0048 & 0.0210 \\
    \hline
\end{tabular}
\captionsetup{width=1\linewidth}
\caption{VCIR process with fractional kernel: varying $T$ and $\Delta t_{k+1}$.}
\label{table:VCIR_T_dt_MOM}
\end{minipage}
\end{table}

\begin{figure}[H]
  \centering
  \begin{minipage}[b]{0.45\textwidth}
    \includegraphics[width=\textwidth]{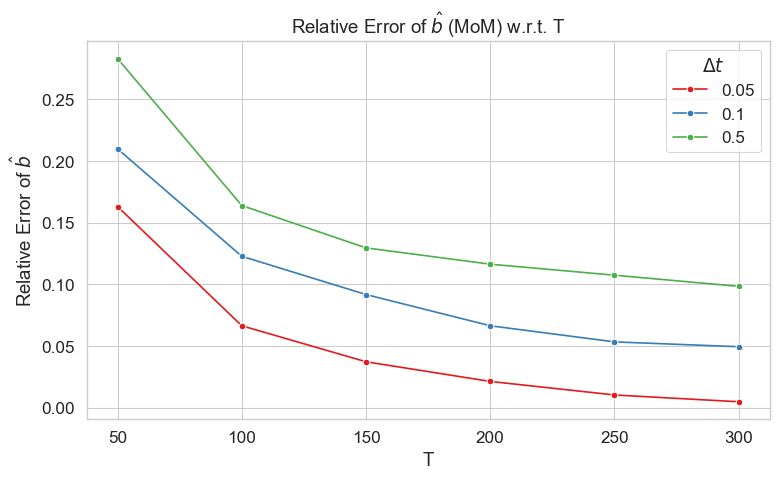}
    \captionsetup{width=1\linewidth}
    \caption{VCIR process with fractional kernel, $\alpha=0.75$, $\Delta t_{k+1}=0.05$.}
    \label{fig:VCIR_b_estimate}
    \end{minipage}
  \hfill
  \begin{minipage}[b]{0.45\textwidth}
    \includegraphics[width=\textwidth]{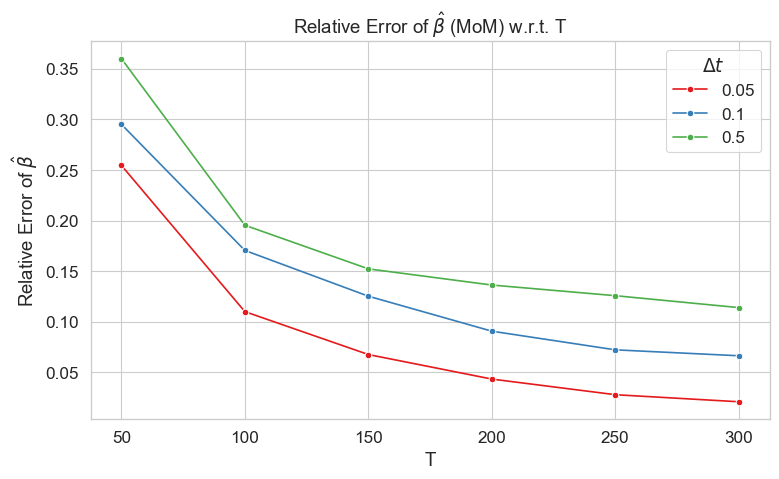}
    \captionsetup{width=1\linewidth}
    \caption{VCIR process with fractional kernel, $\alpha=0.75$, $\Delta t_{k+1}=0.05$.}
    \label{fig:VCIR_beta_estimate}
  \end{minipage}
\end{figure}

\begin{figure}[H]
  \centering
  \begin{minipage}[b]{0.45\textwidth}
    \includegraphics[width=\textwidth]{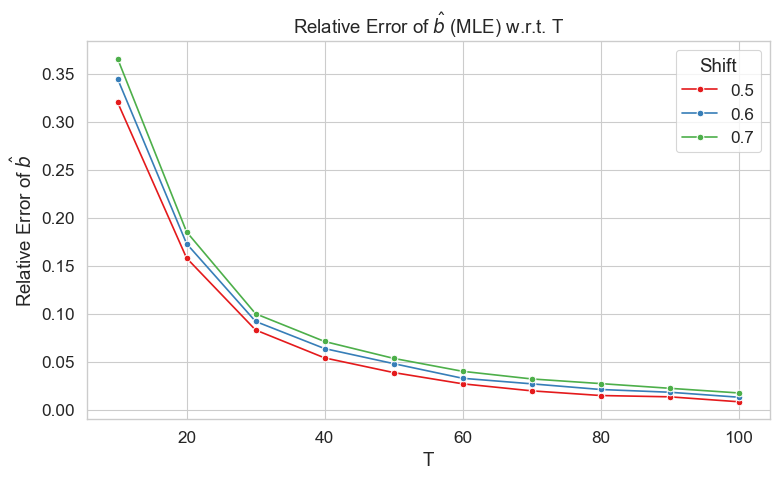}
    \captionsetup{width=1\linewidth}
    \caption{VCIR process with fractional kernel, $\alpha=0.75$, $\Delta t_{k+1} = 0.05$.}
    \label{fig:VCIR_b_MOM_estimate}
    \end{minipage}
  \hfill
  \begin{minipage}[b]{0.45\textwidth}
    \includegraphics[width=\textwidth]{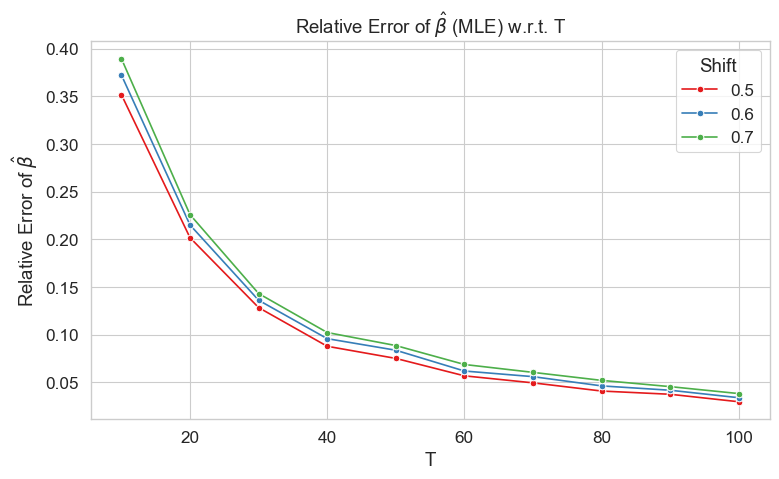}
    \captionsetup{width=1\linewidth}
    \caption{VCIR process with fractional kernel, $\alpha=0.75$, $\Delta t_{k+1} = 0.05$.}
    \label{fig:VCIR_beta_MOM_estimate}
  \end{minipage}
\end{figure}

\begin{table}[H]
\centering
\begin{minipage}[b]{0.4\textwidth}
\begin{tabular}{ |p{1.0cm}||p{2.0cm}|p{2.0cm}|}
 \hline
 \multicolumn{3}{|c|}{$T = 100$, $\varepsilon = 0.15$, and $\alpha = 0.75$} \\
 \hline
 $\Delta t_{k+1}$ & 
 rel. error $\widehat{b}^{\text{MLE}}$ & 
 rel. error $\widehat{\beta}^{\text{MLE}}$ 
 \\
 \hline
1.00 & 0.0000 & 0.0134 \\
0.50 & 0.0000 & 0.0123 \\
0.10 & 0.0000 & 0.0235 \\
0.05 & 0.0006 & 0.0232 \\
0.01 & 0.0054 & 0.0255 \\
     \hline
\end{tabular}
\captionsetup{width=1\linewidth}
\caption{VCIR process with shifted fractional kernel: varying $\Delta t_{k+1}$.}
\label{table:VCIR_dt}
\end{minipage}
  \hfill
\begin{minipage}[b]{0.4\textwidth}
\begin{tabular}{ |p{1.0cm}||p{2.0cm}|p{2.0cm}|}
 \hline
 \multicolumn{3}{|c|}{$T = 100$, $\Delta t_{k+1}$, and $\varepsilon = 0.15$} \\
 \hline
 $\alpha$ & 
 rel. error $\widehat{b}^{\text{MLE}}$ & 
 rel. error $\widehat{\beta}^{\text{MLE}}$ 
 \\
 \hline
0.55 & 0.0007 & 0.0309 \\
0.65 & 0.0006 & 0.0262 \\
0.75 & 0.0006 & 0.0232 \\
0.85 & 0.0007 & 0.0212 \\
0.95 & 0.0007 & 0.0199 \\
     \hline
\end{tabular}
\captionsetup{width=1\linewidth}
\caption{VCIR process with shifted fractional kernel: varying $\alpha$.}
\label{table:VCIR_alpha}
\end{minipage}
\end{table}

\subsection*{Acknowledgement}

M. F. would like to thank the LMRS for the wonderful hospitality and Financial support in 2023 during which a large part of the research was carried out. Moreover, M.F. gratefully acknowledges the Financial support received from Campus France awarded by Ambassade de France en Irlande to carry out the project “Parameter estimation in affine rough models”.

\bibliographystyle{amsplain}
\phantomsection\addcontentsline{toc}{section}{\refname}\bibliography{Bibliography}

\end{document}